\documentclass{amsart}
\usepackage[a4paper, margin=2cm]{geometry}
\usepackage{graphicx}
\usepackage{amssymb,amsmath}
\usepackage{amsthm,thmtools}
\usepackage{mathrsfs}
\usepackage{array}
\usepackage{comment}
\usepackage{mathtools}
\usepackage{tikz}
\usetikzlibrary{arrows,arrows.meta,decorations.pathreplacing,decorations.markings,shapes,calc,cd}
\tikzset{labelsize/.style={font=\scriptsize}}
\tikzset{string/.style={very thick}}
\tikzset{
  pto/.style={->,postaction={decorate},
    decoration={
        markings,
        mark=at position 0.5 with {\arrow{|}}}
  },
}
\tikzset{2cell/.style={-implies,double,double equal sign distance,shorten >=9pt, shorten <=10pt}}
\usepackage[colorlinks=true,pagebackref]{hyperref}
\usepackage{cleveref}

\hypersetup{%
    colorlinks=true,
    linkcolor=violet,
}

\newcolumntype{N}{@{}m{0pt}@{}}

\usepackage{etoolbox}
\makeatletter
\patchcmd{\@counteralias}
 {\@ifdefinable{c@#1}}
 {\expandafter\@ifdefinable\csname c@#1\endcsname}
 {}{}
\makeatother

\numberwithin{equation}{subsection}

\declaretheorem[style=plain,sibling=equation,name=Theorem]{theorem}

\declaretheorem[style=plain,sibling=theorem,name=Proposition]{proposition}
\declaretheorem[style=plain,sibling=theorem,name=Corollary]{corollary}

\declaretheorem[style=definition,qed=$\blacksquare$,sibling=theorem,name=Definition]{definition}
\declaretheorem[style=definition,qed=$\blacksquare$,sibling=theorem,name=Example]{example}
\declaretheorem[style=definition,qed=$\blacksquare$,sibling=theorem,name=Remark]{remark}

\crefname{theorem}{Theorem}{Theorems}
\crefname{section}{Section}{Sections}
\crefname{subsection}{Subsection}{Subsections}
\crefname{definition}{Definition}{Definitions}
\crefname{notation}{Notation}{Notations}
\crefname{example}{Example}{Examples}
\crefname{remark}{Remark}{Remarks}
\crefname{equation}{}{}
\crefname{construction}{Construction}{Constructions}
\crefname{corollary}{Corollary}{Corollaries}
\crefname{proposition}{Proposition}{Propositions}
\crefname{lemma}{Lemma}{Lemmas}

\mathchardef\mhyphen="2D
\newcommand{\Set}{\mathbf{Set}}
\newcommand{\cat}[1]{{\mathcal{#1}}}
\newcommand{\N}{\mathbb{N}}
\newcommand{\G}{\mathbb{G}}
\newcommand{\op}{\mathrm{op}}
\newcommand{\enGph}[1]{{{#1}\mhyphen\mathbf{Gph}}}

\newcommand{\ppair}[1]{\langle #1\rangle} 
\newcommand{\Us}{{U_\mathrm{s}}}

\newcommand{\ar}{\mathop{\mathrm{ar}}}
\newcommand{\StrCats}[1]{{\mathbf{Str}\mhyphen{#1}\mhyphen\mathbf{Cat}_\mathrm{s}}} 
\newcommand{\WkCats}[1]{{\mathbf{Wk}\mhyphen{#1}\mhyphen\mathbf{Cat}_\mathrm{s}}} 
\newcommand{\WkCat}[1]{{\mathbf{Wk}\mhyphen{#1}\mhyphen\mathbf{Cat}}} 
\newcommand{\WkGpds}[1]{{\mathbf{Wk}\mhyphen{#1}\mhyphen\mathbf{Gpd}_\mathrm{s}}} 
\newcommand{\WkGpd}[1]{{\mathbf{Wk}\mhyphen{#1}\mhyphen\mathbf{Gpd}}} 
\newcommand{\GSet}{\mathbf{GSet}}
\newcommand{\id}[3]{\mathrm{id}_{#1}^{#2}(#3)} 
\newcommand{\oldcomp}[3]{\mathbin{{\prescript{}{#1}{\mkern-1.5mu\ast}_{#2}^{#3}}}}  
\newcommand{\comp}[2]{\mathbin{{\ast}_{#1}^{#2}}}  
\newcommand{\fulllabel}{\mathrm{fdl}}
\newcommand{\uum}{\boldsymbol{u^*}}
\newcommand{\uui}{\boldsymbol{u^\mathrm{id}}}
\newcommand{\uuo}{\boldsymbol{u^{p}}}
\newcommand{\one}{\mathbf{1}}
\newcommand{\two}{\mathbf{2}}

\newcommand{\cchi}{\boldsymbol{\chi}}

\newcommand{\phiinv}{\phi^{\mathrm{inv}}}
\newcommand{\uul}{\boldsymbol{\overline{u}}}

\newcommand{\spi}{\mathrm{sp}}
\newcommand{\maini}{{\bar{\imath}}}
\newcommand{\mainj}{{\bar{\jmath}}}
\newcommand{\OO}{G}

\newcommand{\kk}{\boldsymbol{k}}
\newcommand{\uu}{\boldsymbol{u}}

\newcommand{\xx}{\boldsymbol{x}}

\newcommand{\uuinv}{\boldsymbol{u^{\mathrm{inv}}}}
\newcommand{\xxinv}{\boldsymbol{x^{\mathrm{inv}}}}

\newcommand{\uut}{\boldsymbol{\widetilde{u}}}

\newcommand{\kkd}{\boldsymbol{k'}}
\newcommand{\kkdd}{\boldsymbol{k''}}

\newcommand{\pst}{\operatorname{pst}}

\begin{document}
\title{Weakly invertible cells in a weak $\omega$-category}
\author{Soichiro Fujii}
\address{School of Mathematical and Physical Sciences, Macquarie University, NSW 2109, Australia}
\email{s.fujii.math@gmail.com}

\author{Keisuke Hoshino}
  \address{Research Institute for Mathematical Sciences, Kyoto University, Kyoto, Japan}
  \email{hoshinok@kurims.kyoto-u.ac.jp}

\author{Yuki Maehara}
  \address{Research Institute for Mathematical Sciences, Kyoto University, Kyoto, Japan}
  \email{ymaehar@kurims.kyoto-u.ac.jp}

\date{\today}


\begin{abstract}
We study weakly invertible cells in weak $\omega$-categories in the sense of Batanin--Leinster, adopting the coinductive definition of weak invertibility. We show that weakly invertible cells in a weak $\omega$-category are closed under globular pasting. Using this, we generalise elementary properties of weakly invertible cells known to hold in strict $\omega$-categories to weak $\omega$-categories, and show that every weak $\omega$-category has a largest weak $\omega$-subgroupoid.
\end{abstract}

\maketitle

\section{Introduction}
In (higher-dimensional) category theory, the ``correct'' notion of equivalence is often something weaker than that of equality.
When working inside an ordinary category, it is the notion of isomorphism; when working on the totality of all (small) categories, or more generally in a $2$-category, the word ``equivalence'' already has a well-established meaning.
For general $n$-categories where $n \in \mathbb{N}\setminus\{0\}$, one can define such notion of equivalence, or more generally the notion of weakly invertible $k$-cell, by induction: 
weakly invertible $n$-cells in an $n$-category $X$ are the same as strictly invertible $n$-cells, and a $k$-cell $u\colon x\to y$ for $0<k<n$ is weakly invertible if there exists a $k$-cell $v\colon y\to x$ and weakly invertible $(k+1)$-cells $vu\to 1_x$ and $uv\to 1_y$ in $X$. Unravelling the induction, a witness for weak invertibility of a $k$-cell $u\colon x\to y$ in an $n$-category involves a $k$-cell $v\colon y\to x$ together with $2^{m-k+1}$ $m$-cells of suitable types for each $k< m\leq n$, which are subject to $2^{n-k+1}$ equations at the top dimension $n$ (cf.~\emph{exponential wedge} in \cite[Section~1]{Street-oriental}).

In an $\omega$-category, one can still define weakly invertible cells in the same manner, using the unravelled description: a $k$-cell $u\colon x\to y$ in an $\omega$-category $X$ is weakly invertible if there exist a $k$-cell $v\colon y\to x$ and $2^{m-k+1}$ $m$-cells in $X$ for each $m>k$, of suitable types. (This time, we do not demand any equations between these cells.) 
This definition can be stated more succinctly as: a $k$-cell $u\colon x\to y$ in an $\omega$-category $X$ is \emph{weakly invertible} if there exist a $k$-cell $v\colon y\to x$ and \emph{weakly invertible} $(k+1)$-cells $vu\to 1_x$ and $uv\to 1_y$ in $X$. To remove ambiguity in this seemingly circular definition, we should also add that here we are defining weak invertibility \emph{coinductively} (see \cref{rmk:coinduction} for details of conduction). 
This notion of weakly invertible cell in an $\omega$-category plays a key role in e.g.~the definition of a model structure on the category of strict $\omega$-categories \cite{Lafont_Metayer_Worytkiewicz_folk_model_str_omega_cat} (see also \cite{Cheng_dual}). 

In this paper, we study weakly invertible cells in a \emph{weak $\omega$-category} in the sense of Leinster \cite{Leinster_book} (which is based on an earlier definition by Batanin \cite{Batanin_98}). Our main theorem (\cref{thm:main}) states that the set of weakly invertible cells in a weak $\omega$-category is closed under the operations of (globular) pasting.\footnote{We prove this only for globular pasting operations since Leinster's definition of weak $\omega$-category is based on them. Of course, the result trivially extends to whatever notion of pasting operation, as long as it is expressible as a repeated application of globular pasting operations.} For example, if we are given a 2-dimensional pasting diagram 
\begin{equation}
\label{eqn:example-pasting}
\begin{tikzpicture}[baseline=-\the\dimexpr\fontdimen22\textfont2\relax ]
      \node(20) at (0,0) {$a$};
      \node(21) at (2,0) {$b$};
      \node(22) at (4,0) {$c$};
      \node(23) at (6,0) {$d$,};

      \draw [->,bend left=30]  (20) to node[auto, labelsize] {$f$} (21);
      \draw [->,bend right=30] (20) to node[auto, swap,labelsize] {$g$} (21); 
      \draw [->]  (21) to node[auto, labelsize] {$h$} (22);
      \draw [->,bend left=60]  (22) to node[auto, labelsize] {$i$} (23);
      \draw [->]  (22) to node[midway,fill=white,labelsize] {$j$} (23);
      \draw [->,bend right=60]  (22) to node[auto, swap,labelsize] {$k$} (23);

      \draw [->]  (1,0.25) to node[auto,labelsize] {$\alpha$} (1,-0.25);
      \draw [->]  (5,0.58) to node[auto,labelsize] {$\beta$} (5,0.18);
      \draw [->]  (5,-0.18) to node[auto,labelsize] {$\gamma$} (5,-0.58);
\end{tikzpicture}
\end{equation}
in a weak $\omega$-category, in which all $2$-cells $\alpha,\beta$, and $\gamma$ appearing in the diagram are weakly invertible, then so is their composite, regardless of the way of ``bracketing.'' (Recall that in a weak $\omega$-category, composition is neither unital nor associative, so for example the $1$-dimensional domain of the above composite could be taken either as $(ih)f$ or $i(hf)$, and these might be different.) We remark that for the composite $2$-cell of \cref{eqn:example-pasting} to be weakly invertible, the 1-cells in the diagram (such as $f$) need not be weakly invertible. 

The special case of our main theorem, where there is no $n$-cell in an $n$-dimensional pasting diagram, can be regarded as a \emph{coherence} result for weak $\omega$-categories. For example, one can regard \cref{eqn:example-pasting} not as a $2$-dimensional pasting diagram but as a $3$-dimensional one in which no $3$-cell appears. Then, the assumption of the main theorem (i.e., all $3$-cells in \cref{eqn:example-pasting} should be weakly invertible) is vacuously satisfied. Consequently, we see that any two composite $2$-cells $\delta$ and $\delta'$ of \cref{eqn:example-pasting} are connected by a weakly invertible $3$-cell (provided that $\delta$ and $\delta'$ are parallel). See \cref{cor:coherence} and \cref{rmk:contractibility} for more details.

A typical application of our main theorem would proceed as follows.
Firstly, one takes some elementary fact about strict $\omega$-categories (stating either an equality between cells or existence of a weakly invertible cell) which one wishes to generalise to the weak setting.
Then one inserts a coherence cell whenever one sees an equality in the proof of the strict case.
This yields a chain of whiskerings of weakly invertible cells, which can be composed to a single cell.
Finally, using the main theorem, one deduces that the resulting cell is itself weakly invertible.
We give a couple of examples in \cref{cor:unique-inv,cor:invariance}.

Another immediate consequence of the main theorem is that any weak $\omega$-category $X$ has a largest weak $\omega$-subgroupoid $k(X)$ which we call the \emph{core} of $X$. Indeed, let $k(X)$ consist of all (hereditarily) weakly invertible cells of $X$. Then the main theorem ensures that $k(X)$ is a weak $\omega$-subcategory of $X$. We obtain a functor $k\colon \WkCats{\omega}\to\WkGpds{\omega}$ from the category of weak $\omega$-categories to its full subcategory consisting of all weak $\omega$-groupoids, which is characterised as the right adjoint of the inclusion functor $\WkGpds{\omega}\to \WkCats{\omega}$. This generalises the well-known universal property of the core groupoid functor $k\colon\mathbf{Cat}\to\mathbf{Gpd}$. 

\subsection*{Related work}
After presenting our main theorem at Category Theory 2023 in Louvain-la-Neuve, we learnt from Emily Riehl that tslil clingman had independently proved a similar result.
More precisely, Leinster gives a ``non-algebraic'' variant of his definition of weak $\omega$-category in \cite[Definition~L$'$]{Leinster_survey} and \cite[Section~10.2]{Leinster_book}.
In \cite[Section~3.2]{clingman_thesis}, clingman shows that the weakly invertible cells in a slight variant of these {non-algebraic} weak $\omega$-categories (called \emph{proof-relevant categories} in \cite{clingman_thesis}) are closed under compositions.


\subsection*{Outline of the paper}
In \cref{sec:weak-omega-cat} we recall Leinster's definition of weak $\omega$-category and introduce notations. In \cref{sec:main}, after recalling the coinductive definition of weakly invertible cells in a weak $\omega$-category and establishing a few facts about them, we state and prove the main theorem, and discuss a few applications including core weak $\omega$-groupoids of weak $\omega$-categories.

\section{Leinster's definition of weak \texorpdfstring{$\omega$}{omega}-category}
\label{sec:weak-omega-cat}
In this section, we review Leinster's definition of weak $\omega$-category~\cite{Leinster_book}, which is based on an earlier definition by Batanin~\cite{Batanin_98}.
Along the way, we also introduce some notations.

In short, we shall define a weak $\omega$-category as an Eilenberg--Moore algebra of a suitable monad $L$ on the category $\GSet$ of globular sets.
The monad $L$ for weak $\omega$-categories is defined using the monad $T$ on $\GSet$ for \emph{strict}  $\omega$-categories:
$L$ is the \emph{initial cartesian monad over $T$ with contraction}.
We shall explain these notions in order. 
See \cite{Leinster_book} for a more leisurely explanation. (In \cite{Leinster_book}, instead of cartesian monad over $T$, an equivalent notion of $T$-operad is used.)

\subsection{Globular sets}
Let $\G$ be the category freely generated by the graph
\[
\begin{tikzpicture}[baseline=-\the\dimexpr\fontdimen22\textfont2\relax ]
      \node(0) at (0,0) {$[0]$};
      \node(1) at (2,0) {$[1]$};
      \node(d) at (4,0) {$\cdots$};
      \node(n) at (6,0) {$[n]$};
      \node(d2) at (8,0) {$\cdots$};
      
      \draw [->,transform canvas={yshift=3pt}] (0) to node[auto, labelsize] 
      {$\sigma_0$} (1); 
      \draw [->,transform canvas={yshift=-3pt}] (0) to node[auto, 
      swap,labelsize] 
      {$\tau_0$} (1); 
      \draw [->,transform canvas={yshift=3pt}] (1) to node[auto, labelsize] 
      {$\sigma_1$} (d); 
      \draw [->,transform canvas={yshift=-3pt}] (1) to node[auto, 
      swap,labelsize] 
      {$\tau_1$} (d); 
      \draw [->,transform canvas={yshift=3pt}] (d) to node[auto, labelsize] 
      {$\sigma_{n-1}$} (n); 
      \draw [->,transform canvas={yshift=-3pt}] (d) to node[auto, 
      swap,labelsize] 
      {$\tau_{n-1}$} (n); 
      \draw [->,transform canvas={yshift=3pt}] (n) to node[auto, labelsize] 
      {$\sigma_{n}$} (d2); 
      \draw [->,transform canvas={yshift=-3pt}] (n) to node[auto, 
      swap,labelsize] 
      {$\tau_{n}$} (d2);
\end{tikzpicture}
\]
subject to the relations
\[
\sigma_{n+1}\circ \sigma_{n} =\tau_{n+1}\circ \sigma_{n},\qquad \sigma_{n+1}\circ \tau_{n}=\tau_{n+1}\circ \tau_{n} \qquad (\forall n\in \N).
\]
Explicitly, we have
\[
\G\bigl([m],[n]\bigr)=
\begin{cases}
    \{\sigma_m^n,\tau_m^n\} &\text{if $m<n$;}\\
    \{\mathrm{id}_{[n]}\}   &\text{if $m=n$;}\\
    \emptyset       &\text{if $m>n$,}
\end{cases}
\]
where $\sigma_m^n=\sigma_{n-1}\circ \sigma_{n-2}\circ \dots \circ \sigma_{m}$ and $\tau_m^n=\tau_{n-1}\circ \tau_{n-2}\circ \dots \circ \tau_{m}$. 

A \emph{globular set} is a functor $\G^\op\to\Set$, and the category $\GSet$ of globular sets is defined to be the presheaf category $[\G^\op,\Set]$.
Given a globular set $X$, the set $X[n]$ is written as $X_n$ and its elements are called \emph{$n$-cells} of $X$. 
If $m<n$, we denote the function $X\sigma^n_m\colon X_{n}\to X_m$ by $s^X_m$ and $X\tau^n_m$ by $t^X_m$. 
Two $n$-cells $x$ and $y$ of $X$ are said to be \emph{parallel} if $s_{m}^X(x)=s_{m}^X(y)$ and $t_{m}^X(x)=t_{m}^X(y)$ hold for each integer $m$ with $0\leq m<n$.
In other words, all $0$-cells of $X$ are parallel, and when $n\geq 1$, $x,y\in X_n$ are parallel if and only if $s_{n-1}^X(x)=s_{n-1}^X(y)$ and $t_{n-1}^X(x)=t_{n-1}^X(y)$ hold.
For an $n$-cell $u$ of $X$ ($n\geq 1$), we write $u\colon x\to y$ to mean $s_{n-1}^X(u)=x$ and $t_{n-1}^X(u)=y$.

The representable globular set $\G(-,[n])$ is denoted by $\OO^{n}$; it has precisely one $n$-cell, two $m$-cells for $0\leq m<n$ and no $m$-cells for $m>n$. 
For any globular set $X$, the morphisms $\OO^n\to X$ correspond to the $n$-cells of $X$ by the Yoneda lemma; for $x\in X_n$, we also denote the corresponding morphism by $x\colon \OO^n\to X$.
For $m<n$, the natural transformation $\G(-,\sigma_m^n)$ is denoted by $\sigma_m^n\colon \OO^m\to\OO^n$; similarly, $\G(-,\tau_m^n)$ is denoted by $\tau_m^n$.
Let $\partial \OO^n$ be the largest proper globular subset of $\OO^n$; in other words, $\partial \OO^n$ is obtained from $\OO^n$ by removing its unique $n$-cell. 
We denote the associated inclusion by $\iota_n\colon \partial\OO^n\to \OO^n$. 
These morphisms may be depicted as follows:
\[
\begin{tikzpicture}[baseline=-\the\dimexpr\fontdimen22\textfont2\relax ]
      \node(11) at (0,1) {$\bigg( \quad\bigg) $};
      \node(21) at (0,-1) {$\bigg(\bullet\bigg)$};
      
      \draw [->] (0, 0.4) to node[auto,labelsize]{$\iota_0$} (0,-0.4);
\end{tikzpicture}\qquad
\begin{tikzpicture}[baseline=-\the\dimexpr\fontdimen22\textfont2\relax ]
      \node(11) at (0,1) {$\bigg( \bullet$};
      \node(12) at (1.5,1) {$\bullet \bigg)$};
      \node(21) at (0,-1) {$\bigg( \bullet$};
      \node(22) at (1.5,-1) {$\bullet \bigg)$};
      
      \draw [->]  (21) to(22);      
      
      \draw [->] (0.75, 0.4) to  node[auto,labelsize]{$\iota_1$} (0.75,-0.4);
\end{tikzpicture}\qquad
\begin{tikzpicture}[baseline=-\the\dimexpr\fontdimen22\textfont2\relax ]
      \node(11) at (0,1) {$\bigg( \bullet$};
      \node(12) at (1.5,1) {$\bullet \bigg)$};
      \node(21) at (0,-1) {$\bigg( \bullet$};
      \node(22) at (1.5,-1) {$\bullet \bigg)$};
      
      \draw [->,bend left=30]  (11) to node (1u) {} (12);
      \draw [->,bend right=30] (11) to node (1b) {} (12);
      \draw [->,bend left=30]  (21) to node (2u) {} (22);      
      \draw [->,bend right=30] (21) to node (2b) {} (22); 
      
      \draw [->] (2u) to (2b);
      
      \draw [->] (0.75, 0.4) to  node[auto,labelsize]{$\iota_2$} (0.75,-0.4);
\end{tikzpicture}\qquad
\begin{tikzpicture}[baseline=-\the\dimexpr\fontdimen22\textfont2\relax ]
      \node(11) at (0,1) {$\bigg( \bullet$};
      \node(12) at (1.5,1) {$\bullet \bigg)$};
      \node(21) at (0,-1) {$\bigg( \bullet$};
      \node(22) at (1.5,-1) {$\bullet \bigg)$};
      
      \draw [->,bend left=30]  (11) to node (1u) {} (12);
      \draw [->,bend right=30] (11) to node (1b) {} (12);
      \draw [->,bend left=30]  (21) to node (2u) {} (22);      
      \draw [->,bend right=30] (21) to node (2b) {} (22); 
      
      \draw [->,transform canvas={xshift=-0.45em}, bend right=30] (1u) to (1b);
      \draw [->,transform canvas={xshift=0.45em}, bend left=30]  (1u) to (1b);
      \draw [->,transform canvas={xshift=-0.45em}, bend right=30] (2u) to node (3s) {} (2b);
      \draw [->,transform canvas={xshift=0.45em}, bend left=30]  (2u) to node (3t) {} (2b);
      
      \draw [->] (0.6,-1) to (0.9,-1);
      \draw [->] (0.75, 0.4) to  node[auto,labelsize]{$\iota_3$} (0.75,-0.4);
\end{tikzpicture}
\quad \ \dots
\] 
Inductively, $\partial\OO^0$ is the empty (initial) globular set, and for $n\geq 1$, we have a pushout diagram
\begin{equation*}
\begin{tikzpicture}[baseline=-\the\dimexpr\fontdimen22\textfont2\relax ]
      \node(00) at (0,1) {$\partial\OO^{n-1}$};
      \node(01) at (2,1) {$\OO^{n-1}$};
      \node(10) at (0,-1) {$\OO^{n-1}$};
      \node(11) at (2,-1) {$\partial\OO^n$};
      
      \draw [->] (00) to node[auto, labelsize] {$\iota_{n-1}$} (01); 
      \draw [->] (01) to node[auto, labelsize] {} (11); 
      \draw [->] (00) to node[auto,swap,labelsize] {$\iota_{n-1}$} (10); 
      \draw [->] (10) to node[auto,swap,labelsize] {} (11);
\end{tikzpicture}
\end{equation*}
in $\GSet$.
Therefore for any globular set $X$, the morphisms $\partial\OO^n\to X$ correspond to the parallel pairs of $(n-1)$-cells if $n\geq 1$; for each pair $(u,v)$ of parallel $(n-1)$-cells in $X$, we denote the corresponding morphism by $\ppair{u,v}\colon\partial\OO^n\to X$. 

\subsection{The free strict \texorpdfstring{$\omega$}{omega}-category monad \texorpdfstring{$T$}{T}}
\label{subsec:T}
Let $\StrCats{\omega}$ be the category of small strict $\omega$-categories and strict $\omega$-functors.
The forgetful functor $\StrCats{\omega}\to \GSet$ is monadic~\cite[Theorem~F.2.2]{Leinster_book}, and the induced monad on $\GSet$ is denoted by $T=(T,\eta^T,\mu^T)$. 

As the definition of the monad $L$ for weak $\omega$-categories depends on $T$, let us investigate the structure of the monad $T$.
To this end, it is helpful to use the following notion.
\begin{definition}[{See \cite[Section~2.1]{Weber-thesis} and \cite[Section~4]{Weber-generic}}]
    A (globular) \emph{pasting scheme} is a table (i.e., a finite sequence) of non-negative integers
    \[
    \kk = \begin{bmatrix}
    k_0 & & k_1 & & \dots & & k_r\\
    & \underline{k}_1 & & \underline{k}_2 & \dots & \underline{k}_r &
    \end{bmatrix}
    \]
    with $r\ge 0$ and $k_{i-1} > \underline k_i < k_i$ for all $1 \le i \le r$. We call $r$ the \emph{rank} of $\kk$.
    For $n\ge 0$, a \emph{pasting scheme of dimension $n$} is a pasting scheme $\kk$ as above which moreover satisfies $k_i \le n$ for all $0 \le i \le r$.
    
    Let $X$ be a globular set and $\kk$ be a pasting scheme of rank $r$.
    A \emph{pasting diagram of shape $\kk$ in $X$} is a table
    \[
    \uu = \begin{bmatrix}
    u_0 & & u_1 & & \dots & & u_r\\
    & \underline{u}_1 & & \underline{u}_2 & \dots & \underline{u}_r &
    \end{bmatrix}
    \]
    of cells $u_i \in X_{k_i}$ for $0 \le i \le r$ and $\underline u_i \in X_{\underline k_i}$ for $1 \le i \le r$, such that
    \[
    t_{\underline k_i}^X(u_{i-1}) =\underline u_i = s_{\underline k_i}^X(u_i)
    \]
    for all $1 \le i \le r$.
    The pasting diagram $\uu$ is \emph{of dimension $n$} (resp.~\emph{of rank $r$}) if its shape $\kk$ is so.
\end{definition}
Each pasting scheme
    \[
    \kk = \begin{bmatrix}
    k_0 & & k_1 & & \dots & & k_r\\
    & \underline k_1 & & \underline k_2 & \dots & \underline k_r &
    \end{bmatrix}
    \]
has an associated globular set $\widehat{\kk}$ called its \emph{realisation} (see e.g.~\cite[Section~8.1]{Leinster_book}, \cite[Section~2.1]{Weber-thesis}, and \cite[Section~4]{Weber-generic}). It is defined as the colimit of the diagram 
\begin{equation*}
\begin{tikzpicture}[baseline=-\the\dimexpr\fontdimen22\textfont2\relax ]
      \node(00) at (0,0.75) {$\OO^{k_0}$};
      \node(01) at (3,0.75) {$\OO^{k_1}$};
      \node(02) at (9.5,0.75) {$\OO^{k_r}$};
      \node(10) at (1.5,-0.75) {$\OO^{\underline k_1}$};
      \node(11) at (4.5,-0.75) {$\OO^{\underline k_2}$};
      \node(12) at (8,-0.75) {$\OO^{\underline k_r}$}; 
      \node at (6.25,0) {$\dots$};
      \node[rotate=45] (a) at (5.5,0.25) {$\dots$};
      \node[rotate=-45] (a) at (7,0.25) {$\dots$};
      
      \draw [->] (10) to node[auto, labelsize] {$\tau_{\underline k_1}^{k_0}$} (00); 
      \draw [->] (10) to node[auto, swap, labelsize] {$\sigma_{\underline k_1}^{k_1}$} (01); 
      \draw [->] (11) to node[auto,labelsize] {$\tau_{\underline k_2}^{k_1}$} (01);
      \draw (11) to node[auto,labelsize] {} (5.25,0);
      \draw (12) to node[auto,labelsize] {} (7.25,0);
      \draw [->] (12) to node[auto,swap,labelsize] {$\sigma_{\underline k_r}^{k_r}$} (02);
\end{tikzpicture}
\end{equation*}
in $\GSet$.
The pasting diagrams of shape $\kk$ in $X$ correspond to the morphisms of globular sets $\widehat{\kk}\to X$.
The following are examples of pasting schemes and their realisations. 
\begin{equation*}
\label{eqn:globular_pasting_shcemes}
\begin{tikzpicture}[baseline=-\the\dimexpr\fontdimen22\textfont2\relax ]
      \node(21) at (0,0) {$\bullet$};
      \node(l1) at (0,1.5) {$\begin{bmatrix}
    0\\
    \ 
    \end{bmatrix}$};
\end{tikzpicture}
\qquad
\begin{tikzpicture}[baseline=-\the\dimexpr\fontdimen22\textfont2\relax ]
      \node(31) at (-0.5,0) {$\bullet$};
      \node(32) at (0.75,0) {$\bullet$};
      \node(33) at (2,0) {$\bullet$};
      \node(l0) at (0.75,1.5) {$\begin{bmatrix}
    1 & & 1\\
    & 0 & 
    \end{bmatrix}$};
      
      \draw [->]  (31) to (32);
      \draw [->] (32) to  (33);
\end{tikzpicture}
\qquad
\begin{tikzpicture}[baseline=-\the\dimexpr\fontdimen22\textfont2\relax ]
      \node(20) at (-1.5,0) {$\bullet$};
      \node(21) at (0,0) {$\bullet$};
      \node(22) at (1.5,0) {$\bullet$};
      \node(23) at (3,0) {$\bullet$};
      \node(l0) at (0.75,1.5) {$\begin{bmatrix}
    2 & & 1 & & 2 & & 2\\
    & 0 & & 0 & & 1 & 
    \end{bmatrix}$};
      
      \draw [->]  (21) to (22);
      
      \draw [->,bend left=30]  (20) to node (2u) {} (21);      
      \draw [->,bend right=30] (20) to node (2b) {} (21); 
      \draw [->] (2u) to (2b);

      \draw [->,bend left=50]  (22) to node (3u) {} (23);
      \draw [->] (22) to node (3m) {} (23);
      \draw [->,bend right=50] (22) to node (3b) {} (23);
      \draw [->] (3u) to (3m);
      \draw [->] (3m) to (3b);
\end{tikzpicture}
\end{equation*}

Let $1$ be the terminal globular set. Then $T1$ is the underlying globular set of the free strict $\omega$-category over $1$.
We claim that the set $(T1)_n$ of all $n$-cells of $T1$ can be identified with the set of all pasting schemes of dimension $n$. Indeed, by \cite[Proposition~F.2.3]{Leinster_book}, $(T1)_n$ can be described inductively as:
\begin{itemize}
    \item $(T1)_0$ is a singleton, and
    \item for $n>0$, $(T1)_n$ is the underlying set of the free monoid on $(T1)_{n-1}$.
\end{itemize} 
Writing the unique element of $(T1)_0$ as $[\ ]$ and an element of the free monoid on a set $A$ as a list $[a_1,\dots,a_n]$ of elements $a_1,\dots,a_n\in A$, we have for example $[[[\ ]],[\ ],[[\ ],[\ ]]]\in (T1)_2$; in general, $(T1)_n$ consists of all such iterated lists of depth at most $n+1$. The elements of  $(T1)_n$ can be equivalently described by a sequence of integers $[m_0,m_1,\dots,m_l]$ with
\begin{itemize}
    \item $l\geq 2$,
    \item $m_0=m_l=-1$,
    \item $|m_i-m_{i-1}|=1$ for all $1\leq i\leq l$, and
    \item $0\leq m_i\leq n$ for all $1\leq i\leq l-1$.
\end{itemize}
Indeed, given an iterated list, one can produce such a sequence of integers starting from $-1$ by reading the characters of the list from left to right under the following rules: add $1$ for each opening bracket $[$ and subtract $1$ for each closing bracket $]$ (and ignore the commas). For example, $[[[\ ]],[\ ],[[\ ],[\ ]]]$ turns into the sequence $[-1,0,1,2,1,0,1,0,1,2,1,2,1,0,-1]$. Such a sequence of integers $[m_0,m_1,\dots,m_l]$ (which corresponds to a \emph{smooth zig-zag sequence} in \cite{Weber-thesis,Weber-generic}) can in turn be reconstructed from its subsequence consisting of the inner minimal and maximal elements, i.e., $m_i$ with $0<i<l$ and $m_{i-1}=m_{i+1}$. Thus $[-1,0,1,2,1,0,1,0,1,2,1,2,1,0,-1]$ corresponds to the pasting scheme (which is called a \emph{zig-zag sequence} in \cite{Weber-thesis,Weber-generic}) 
\[
\begin{bmatrix}
    2 & & 1 & & 2 & & 2\\
    & 0 & & 0 & & 1 & 
    \end{bmatrix}
\]
we saw earlier.

Note that a single pasting scheme $\kk$ appears as cells in $T1$ of different dimensions. For clarity, we sometimes write $\kk^{(n)}$ when $\kk$ is being regarded as an $n$-cell of $T1$.
We say that an $n$-cell $\kk$ of $T1$ is \emph{degenerate} if $\kk$ is also an $(n-1)$-cell of $T1$. Thus
    \[
    \kk = \begin{bmatrix}
    k_0 & & k_1 & & \dots & & k_r\\
    & \underline k_1 & & \underline k_2 & \dots & \underline k_r &
    \end{bmatrix}
    \]
is degenerate as an $n$-cell of $T1$ if and only if $\max\{k_0,\dots, k_r\}<n$. Otherwise, $\kk$ is a \emph{non-degenerate} $n$-cell of $T1$.
In other words, a cell in $T1$ is degenerate if and only if it is an identity cell in the free strict $\omega$-category over $1$.

For each $0\leq m<n$, the source and target maps $s^{T1}_m\colon (T1)_n\to(T1)_m$ and $t^{T1}_m\colon (T1)_n\to(T1)_m$ of the globular set $T1$ are equal, and are defined as follows.
Given a pasting scheme $\kk$ of dimension $n$ and of rank $r$, and an integer $m$ with $0\leq m<n$, an \emph{$m$-transversal component}  of $\kk$ is a subsequence of $\kk$ of the form 
    \begin{equation}
    \label{eqn:m-triv}
    \begin{bmatrix}
    k_{i} & & k_{i+1} & & \dots & & k_{j}\\
    & \underline k_{i+1} & & \underline k_{i+2} &\dots & \underline k_{j} &
    \end{bmatrix}
    \end{equation}
with $0 \le i \le j \le r$ such that
\begin{itemize}
    \item $k_{{l}}>m$ for all $i\le {{l}}\le j$,
    \item $\underline k_{{l}} \geq m$ for all $i+1 \leq {{l}} \le j$,
    \item either $i = 0$ or $\underline k_{i} < m$, and
    \item either $j = r$ or $\underline k_{j+1} < m$.
\end{itemize}
(The first clause is not completely redundant because of the case $i=j$.)
We also use the phrase ``$m$-transversal component $0\leq i\leq j\leq r$ of $\kk$'' to refer to the subsequence \cref{eqn:m-triv}.
For $\kk\in (T1)_n$, the pasting scheme $s^{T1}_m(\kk)=t^{T1}_m(\kk)\in(T1)_m$ is obtained from $\kk$ by replacing each of its $m$-transversal components with the sequence 
    \[
    \begin{bmatrix}
    m \\
    \ 
    \end{bmatrix}.
    \]
For example, we have
\[
s^{T1}_4\left(
\begin{bmatrix}
    3 & & 6 & & 5 & & 7 & & 2 & & 6\\
    & 2 & & 3 & & 4 & & 0 & & 1 &
\end{bmatrix} \right)= 
\begin{bmatrix}
    3 & & 4 & & 4 & & 2 & & 4 \\
    & 2 & & 3 & & 0 & & 1 & 
\end{bmatrix}.
\]

For any globular set $X$ and $n\geq 0$, the set $(TX)_n$ can be identified with the set of all pasting diagrams in $X$ of dimension $n$ by \cite[Proposition~F.2.3]{Leinster_book}. For each $0\leq m<n$, the source map $s^{TX}_m\colon (TX)_n\to (TX)_m$ is defined as follows. Let 
    \[
    \uu = \begin{bmatrix}
    u_0 & & u_1 & & \dots & & u_r\\
    & \underline u_1 & & \underline u_1 & \dots & \underline u_r &
    \end{bmatrix}
    \]
be an element of $(TX)_n$, with underlying pasting scheme $\kk$. Then $s^{TX}_m(\uu)\in(TX)_m$ is obtained from $\uu$ by, for each of the $m$-transversal components $0\leq i\leq j\leq r$ of $\kk$, replacing the corresponding subsequence 
    \begin{equation}
    \label{eqn:m-transversal-u-x}
    \begin{bmatrix}
    u_{i} & & u_{i+1} & & \dots & & u_{j}\\
    & \underline u_{i+1} & & \underline u_{i+2} & \dots & \underline u_{j} &
    \end{bmatrix}
    \end{equation}
of $\uu$ with the sequence 
    \[
    \begin{bmatrix}
    s^X_m(u_{i})\\
    \ 
    \end{bmatrix}.
    \]
Similarly, $t^{TX}_m(\uu)\in(TX)_m$ is obtained from $\uu$ by replacing each instance of \cref{eqn:m-transversal-u-x} with 
    \[
    \begin{bmatrix}
    t^X_m(u_{j})\\
    \ 
    \end{bmatrix}.
    \]
For example, given a $2$-dimensional pasting diagram 
\[
\uu=
\begin{bmatrix}
    \alpha & & h & & \beta & & \gamma\\
    & b & & c & & j & 
\end{bmatrix}
\]
in $X$, which may be depicted as 
\[
\begin{tikzpicture}[baseline=-\the\dimexpr\fontdimen22\textfont2\relax ]
      \node(20) at (0,0) {$a$};
      \node(21) at (2,0) {$b$};
      \node(22) at (4,0) {$c$};
      \node(23) at (6,0) {$d$,};

      \draw [->,bend left=30]  (20) to node[auto, labelsize] {$f$} (21);
      \draw [->,bend right=30] (20) to node[auto, swap,labelsize] {$g$} (21); 
      \draw [->]  (21) to node[auto, labelsize] {$h$} (22);
      \draw [->,bend left=60]  (22) to node[auto, labelsize] {$i$} (23);
      \draw [->]  (22) to node[midway,fill=white,labelsize] {$j$} (23);
      \draw [->,bend right=60]  (22) to node[auto, swap,labelsize] {$k$} (23);

      \draw [->]  (1,0.25) to node[auto,labelsize] {$\alpha$} (1,-0.25);
      \draw [->]  (5,0.58) to node[auto,labelsize] {$\beta$} (5,0.18);
      \draw [->]  (5,-0.18) to node[auto,labelsize] {$\gamma$} (5,-0.58);
\end{tikzpicture}
\]
we have 
\begin{align*}
s^{TX}_1(\uu)&=
\begin{bmatrix}
    s^X_1(\alpha) & & h & & s^X_1(\beta)\\
    & b & & c &  
\end{bmatrix}=
\begin{bmatrix}
    f & & h & & i\\
    & b & & c &  
\end{bmatrix},
\end{align*}
which may be depicted as
\[
\begin{tikzpicture}[baseline=-\the\dimexpr\fontdimen22\textfont2\relax ]
      \node(20) at (0,0) {$a$};
      \node(21) at (2,0) {$b$};
      \node(22) at (4,0) {$c$};
      \node(23) at (6,0) {$d$,};

      \draw [->]  (20) to node[auto, labelsize] {$f$} (21);
      \draw [->]  (21) to node[auto, labelsize] {$h$} (22);
      \draw [->]  (22) to node[auto, labelsize] {$i$} (23);
\end{tikzpicture}
\]
and 
\begin{align*}
t^{TX}_1(\uu)&=
\begin{bmatrix}
    t^X_1(\alpha) & & h & & t^X_1(\gamma)\\
    & b & & c &  
\end{bmatrix}=
\begin{bmatrix}
    g & & h & & k\\
    & b & & c &  
\end{bmatrix},
\end{align*}
which may be depicted as
\[
\begin{tikzpicture}[baseline=-\the\dimexpr\fontdimen22\textfont2\relax ]
      \node(20) at (0,0) {$a$};
      \node(21) at (2,0) {$b$};
      \node(22) at (4,0) {$c$};
      \node(23) at (6,0) {$d$,};

      \draw [->]  (20) to node[auto, labelsize] {$g$} (21);
      \draw [->]  (21) to node[auto, labelsize] {$h$} (22);
      \draw [->]  (22) to node[auto, labelsize] {$k$} (23);
\end{tikzpicture}
\]
as expected.
    
The action of $T$ on morphisms of globular sets is straightforward: if $f\colon X\to Y$ is a morphism of globular sets, then $Tf\colon TX\to TY$ is a morphism of globular sets which maps each pasting diagram 
    \[
    \uu = \begin{bmatrix}
    u_0 & & u_1 & & \dots & & u_r\\
    & \underline u_1 & & \underline u_2 & \dots & \underline u_r &
    \end{bmatrix}
    \]
in $X$ to the pasting diagram 
    \[
    (Tf)(\uu) = \begin{bmatrix}
    fu_0 & & fu_1 & & \dots & & fu_r\\
    & f\underline u_1 & & f\underline u_2 & \dots & f\underline u_r &
    \end{bmatrix}
    \]
in $Y$. In particular, for the unique morphism $!\colon X\to 1$ to the terminal globular set $1$, the induced morphism $T!\colon TX\to T1$ maps each pasting diagram $\uu$ in $X$ to its shape $\kk$.

The unit $\eta^T_X\colon X\to TX$ maps each $n$-cell $x$ of $X$ to the $n$-cell 
    \[
    \begin{bmatrix}
    x\\
    \ 
    \end{bmatrix}
    \]
of $TX$, which we write as $[x]$. In particular, $\eta^T_1\colon 1\to T1$ maps the unique $n$-cell of $1$ to the $n$-cell 
\begin{equation*}
    \begin{bmatrix}
    n\\
    \ 
    \end{bmatrix}
\end{equation*}
of $T1$, which we write as $[n]$. We have $[n]\colon [n-1]\to [n-1]$ for each $n\geq 1$.

A cell of $T^2X$ is a pasting diagram in $TX$, namely a table of pasting diagrams in $X$, such as 
\[
\overline{\uu}=
    \begin{bmatrix}
    \uu_0 & &\uu_1 & &  \dots & & \uu_r\\
    & \underline \uu_1 & & \underline \uu_2 & \dots & \underline \uu_r &
    \end{bmatrix}.
\]
The multiplication $\mu^T_X\colon T^2X\to TX$ maps $\overline{\uu}$ to a single pasting diagram $\mu^T_X(\overline{\uu})$ in $X$, which is obtained by suitably ``gluing together'' the pasting diagrams $    \uu_i$ along $\underline \uu_i$. 

An Eilenberg--Moore algebra of the monad $T$ consists of a globular set $X$ together with a structure map $\gamma\colon TX\to X$ satisfying the usual axioms. $\gamma$ maps each $n$-dimensional pasting diagram $\uu$ in $X$ to an $n$-cell $\gamma(\uu)$ of $X$, namely the (pasting) composite of $\uu$. 
It captures the notion of strict $\omega$-category by requiring that each pasting diagram in it should admit a unique composite.

\subsection{The idea of weak \texorpdfstring{$\omega$}{omega}-category}
\label{subsec:idea_of_weak_omega_cat}
In this subsection, we shall give a heuristic discussion that motivates the definition of the monad $L$ (on $\GSet$) for weak $\omega$-categories; an actual definition is carried out in the next subsection. 
First note that, whatever we define $L$ to be, there should be a monad morphism $\ar^{L} \colon L \to T$ since we would want strict $\omega$-categories to be a special case of weak $\omega$-categories.
(Here the notation $\ar^{L}$ stands for \emph{arity}, and this terminology will be justified below where we consider its component $\ar^{L}_1$ at the terminal globular set $1$.)

Suppose that a globular set $X$ contains a sequence of $1$-cells:
\[
\begin{tikzpicture}[baseline=-\the\dimexpr\fontdimen22\textfont2\relax ]
	\node(20) at (0,0) {$a$};
	\node(21) at (2,0) {$b$};
	\node(22) at (4,0) {$c$};
	\node(23) at (6,0) {$d$.};
	
	\draw [->]  (20) to node[auto, labelsize] {$f$} (21);
	\draw [->]  (21) to node[auto, labelsize] {$g$} (22);
	\draw [->]  (22) to node[auto, labelsize] {$h$} (23);
\end{tikzpicture}
\]
In contrast to the strict case where we would have a unique composite $hgf$ in $TX$, the free weak $\omega$-category $LX$ should contain \emph{distinct} $1$-cells $(hg)f$ and $h(gf)$ (and also other composites with identities inserted in various positions).
Motivated by such examples, we claim that $(LX)_n$ should consist of the \emph{pasting instructions} $(\phi,\uu)$ of dimension $n$ in $X$, by which we mean a pasting diagram $\uu\in (TX)_n$ of dimension $n$ in $X$ together with an additional piece of data $\phi$ that encodes \emph{how} to compose the pasting diagram $\uu$. 
Since it should be the \emph{shape} of the pasting diagram $\uu$ (rather than the labels $u_i$ and $\underline u_i$) that determines what $\phi$ can be, we may simply take $\phi$ to be an element of $(L1)_n$ with appropriate arity; that is, we want the globular set $LX$ to be the pullback 
\begin{equation}
	\label{eqn:ar-L-cartesian}
	\begin{tikzpicture}[baseline=-\the\dimexpr\fontdimen22\textfont2\relax ]
		\node(00) at (0,1) {$LX$};
		\node(01) at (2,1) {$L1$};
		\node(10) at (0,-1) {$TX$};
		\node(11) at (2,-1) {$T1$,};
		
		\draw [->] (00) to node[auto, labelsize] {$L!$} (01); 
		\draw [->] (01) to node[auto, labelsize] {$\ar^{L}_1$} (11); 
		\draw [->] (00) to node[auto,swap,labelsize] {$\ar^{L}_X$} (10); 
		\draw [->] (10) to node[auto,swap,labelsize] {$T!$} (11);
 \end{tikzpicture}
\end{equation}
where $\ar^{L}\colon L\to T$ is the monad morphism mentioned earlier.
{We also call elements $\phi$ of $(L1)_n$ \emph{pasting instructions}.}

Thus, a major part of the definition of the monad $L$ is reduced to describing what $L1$ is.
Recall that, in the previous subsection, we saw a description of the strict $\omega$-categories as those globular sets equipped with a \emph{unique} composite for each pasting diagram.
The structure of $L1$ should somehow encode a weak version of this, so we want each pasting diagram to have \emph{some} composite, which should moreover be unique up to suitably invertible higher cell.
To make this precise, we draw insight from the \emph{pasting theorem} in the $2$-dimensional case.

The pasting theorem for \emph{strict} $2$-categories \cite{Power_2-cat_pasting} tells us that a pasting diagram such as
\begin{equation*}
	\begin{tikzpicture}[baseline=-\the\dimexpr\fontdimen22\textfont2\relax ]
		\node(20) at (0,0) {$a$};
		\node(21) at (2,0) {$b$};
		\node(22) at (4,0) {$c$};
		\node(23) at (6,0) {$d$};
		
		\draw [->,bend left=30]  (20) to node[auto, labelsize] {$f$} (21);
		\draw [->,bend right=30] (20) to node[auto, swap,labelsize] {$g$} (21); 
		\draw [->]  (21) to node[auto, labelsize] {$h$} (22);
		\draw [->,bend left=60]  (22) to node[auto, labelsize] {$i$} (23);
		\draw [->]  (22) to node[midway,fill=white,labelsize] {$j$} (23);
		\draw [->,bend right=60]  (22) to node[auto, swap,labelsize] {$k$} (23);
		
		\draw [->]  (1,0.25) to node[auto,labelsize] {$\alpha$} (1,-0.25);
		\draw [->]  (5,0.58) to node[auto,labelsize] {$\beta$} (5,0.18);
		\draw [->]  (5,-0.18) to node[auto,labelsize] {$\gamma$} (5,-0.58);
	\end{tikzpicture}
\end{equation*}
(or much more complicated ones) admits a unique composite $2$-cell $ihf \to khg$.
In the \emph{weak} case of bicategories \cite[Appendix~A]{Verity-thesis}, however, we must first specify how to interpret ``$ihf$'' and ``$khg$''; once we have fixed the bracketing on these $1$-cells, we obtain a unique composite $2$-cell.
When defining $L$, we can follow the same pattern at least for the existence part: \emph{given a pasting scheme of dimension $n$ together with pasting instructions on its $(n-1)$-dimensional source and target, it must be possible to extend them to a pasting instruction on the whole pasting scheme.}

More precisely, we ask that $\ar^{L}_1 \colon L1 \to T1$ have the right lifting property with respect to the boundary inclusion $\iota_n\colon\partial\OO^n\to\OO^n$ for each $n\geq 1$ (cf.~\cite{Garner_univ}). 
This means that given any $u$ and $v$ making the outer square in the following diagram commutative, there exists a (not necessarily unique) morphism $w$ making both triangles in the diagram commutative: 
\begin{equation*}
	\begin{tikzpicture}[baseline=-\the\dimexpr\fontdimen22\textfont2\relax ]
		\node(00) at (0,1) {$\partial\OO^n$};
		\node(01) at (2,1) {$L1$};
		\node(10) at (0,-1) {$\OO^n$};
		\node(11) at (2,-1) {$T1$.};
		
		\draw [->] (00) to node[auto, labelsize] {$u$} (01); 
		\draw [->] (01) to node[auto, labelsize] {$\ar^{L}_1$} (11); 
		\draw [->] (00) to node[auto,swap,labelsize] {$\iota_n$} (10); 
		\draw [->] (10) to node[auto,swap,labelsize] {$v$} (11);   
		\draw [->] (10) to node[midway,fill=white,labelsize] {$w$} (01);
	\end{tikzpicture}
\end{equation*}

Note that, thanks to the pullback condition \cref{eqn:ar-L-cartesian}, the component $\ar^{L}_X$ at any globular set $X$ inherits the right lifting property from $\ar^{L}_1$; given a commutative square as below left, we can first find a lift of the outer square using the right lifting property of $\ar^{L}_1$, and then use the universal property of the pullback to obtain the desired lift:
\begin{equation*}
	\begin{tikzpicture}[baseline=-\the\dimexpr\fontdimen22\textfont2\relax ]
		\node(00) at (0,1) {$\partial\OO^n$};
		\node(01) at (2,1) {$LX$};
		\node(02) at (4,1) {$L1$};
		\node(10) at (0,-1) {$\OO^n$};
		\node(11) at (2,-1) {$TX$};
		\node(12) at (4,-1) {$T1$.};
		
		\draw [->] (00) to node[auto, labelsize] {$u$} (01); 
		\draw [->] (01) to node[auto,near end, labelsize] {$\ar^{L}_X$} (11); 
		\draw [->] (02) to node[auto, labelsize] {$\ar^{L}_1$} (12); 
		\draw [->] (00) to node[auto,swap,labelsize] {$\iota_n$} (10); 
		\draw [->] (10) to node[auto,swap,labelsize] {$v$} (11);   
		\draw [->, dashed] (10) to (01);
		\draw [->, dashed] (10) to (02);
		\draw [->] (01) to node[auto, labelsize] {$L!$} (02);
		\draw [->] (11) to node[auto, labelsize, swap] {$T!$} (12);
	\end{tikzpicture}
\end{equation*}

Thus we have taken care of the existence part of the pasting theorem, and in fact formalising these ideas (combined with a suitable universal property) essentially leads to the precise definition given in the next subsection.
The uniqueness part turns out to be a theorem (see \cref{cor:coherence} and \cref{rmk:contractibility}) rather than part of the definition.

\subsection{The definition of weak \texorpdfstring{$\omega$}{omega}-category}
A natural transformation is \emph{cartesian} if each of its naturality squares is cartesian (i.e., a pullback square). A monad $(S,\eta^S,\mu^S)$ on a category with pullbacks is \emph{cartesian} if its functor part $S$ preserves all pullbacks and its unit $\eta^S$ and multiplication $\mu^S$ are cartesian natural transformations. 
If $S$ and $S'$ are monads on a category with pullbacks, then a monad morphism $\alpha\colon S\to S'$ is \emph{cartesian} if it is cartesian as a natural transformation. 
Notice that if $\cat{E}$ is a category with pullbacks, $S$ is a monad on $\cat{E}$, $S'$ is a cartesian monad on $\cat{E}$, and $\alpha\colon S\to S'$ is a cartesian monad morphism, then $S$ is necessarily a cartesian monad.

It is known that the monad $T$ on $\GSet$ for strict $\omega$-categories is cartesian \cite[Theorem~F.2.2]{Leinster_book}.
A \emph{cartesian monad over $T$} is a (necessarily cartesian) monad $P=(P,\eta^P,\mu^P)$ on $\GSet$ equipped with a cartesian monad morphism $\ar^{P}\colon P\to T$. 
A \emph{cartesian monad over $T$ with contraction} is a cartesian monad $(P,\ar^{P})$ over $T$ equipped with a function $\kappa^P$ which assigns for each $n\geq 1$, $u\colon \partial \OO^n\to P1$, and $v\colon \OO^n\to T1$ such that ${\ar^{P}_1}\circ u=v\circ \iota_n$, a morphism $\kappa^P(u,v)\colon\OO^n\to P1$ such that $u=\kappa^P(u,v)\circ \iota_n$ and $v={\ar^{P}_1}\circ \kappa^P(u,v)$:
\begin{equation*}
\begin{tikzpicture}[baseline=-\the\dimexpr\fontdimen22\textfont2\relax ]
      \node(00) at (0,1) {$\partial\OO^n$};
      \node(01) at (2,1) {$P1$};
      \node(10) at (0,-1) {$\OO^n$};
      \node(11) at (2,-1) {$T1$.};
      
      \draw [->] (00) to node[auto, labelsize] {$u$} (01); 
      \draw [->] (01) to node[auto, labelsize] {$\ar^{P}_1$} (11); 
      \draw [->] (00) to node[auto,swap,labelsize] {$\iota_n$} (10); 
      \draw [->] (10) to node[auto,swap,labelsize] {$v$} (11);   
      \draw [->] (10) to node[midway,fill=white,labelsize] {$\kappa^P(u,v)$} (01);
\end{tikzpicture}
\end{equation*}
      

Now define the category $\cat{C}$ of cartesian monads over $T$ with contraction as follows. An object of $\cat{C}$ is a cartesian monad over $T$ with contraction $(P,\ar^{P},\kappa^P)$, and a morphism $(P,\ar^{P},\kappa^P)\to (Q,\ar^{Q},\kappa^Q)$ in $\cat{C}$ is a (necessarily cartesian) monad morphism $\alpha\colon P\to Q$ such that $\ar^{P}={\ar^{Q}}\circ \alpha$ and that preserves the contractions, in the sense that following diagram commutes (i.e., $\alpha_1\circ \kappa^P(u,v)=\kappa^Q(\alpha_1\circ u,v)$ holds) for each $n\geq 1$, $u\colon \partial \OO^n\to P1$, and $v\colon \OO^n\to T1$ such that ${\ar^{P}_1}\circ u=v\circ \iota_n$:
\begin{equation*}
\begin{tikzpicture}[baseline=-\the\dimexpr\fontdimen22\textfont2\relax ]
      \node(00) at (0,2) {$\partial \OO^n$};
      \node(01) at (2,2) {$P1$};
      \node(10) at (0,0) {$\OO^n$};
      \node(11) at (2,0) {$T1$};
      \node(02) at (5,2) {$Q1$};
      \node(12) at (5,0) {$T1$.};
      
      \draw [->] (00) to node[auto, labelsize] {$u$} (01); 
      \draw [->] (01) to node[auto, labelsize] {$\alpha_1$} (02); 
      \draw [->] (01) to node[auto,near start, labelsize] {$\ar^{P}_1$} (11); 
      \draw [->] (02) to node[auto, labelsize] {$\ar^{Q}_1$} (12); 
      \draw [->] (11) to node[auto,swap,labelsize] {$\mathrm{id}$} (12);  
      \draw [->] (00) to node[auto,swap,labelsize] {$\iota_n$} (10); 
      \draw [->] (10) to node[auto,swap,labelsize] {$v$} (11);   
      \draw [->] (10) to node[midway,fill=white,labelsize] {$\kappa^P(u,v)$} (01);
      \draw [->] (10) to node[near end,fill=white,labelsize] {$\kappa^Q(\alpha_1\circ u,v)$} (02);
\end{tikzpicture}\qedhere
\end{equation*}

Using the alternative description of cartesian monads over $T$ as \emph{$T$-operads}, one can show that $\cat{C}$ has an initial object $(L,\ar^{L},\kappa^L)=(L,\eta^L,\mu^L,\ar^{L},\kappa^L)$ \cite[Corollary~G.1.2]{Leinster_book}. 

\begin{definition}[{\cite{Leinster_book}}]
    A \emph{weak $\omega$-category} is an Eilenberg--Moore algebra  $(X,\xi\colon LX\to X)$ of the monad $L=(L,\eta^L,\mu^L)$.
\end{definition}

We define the category $\WkCats{\omega}$ of weak $\omega$-categories and \emph{strict $\omega$-functors} between them as the Eilenberg--Moore category of the monad $L$. (See \cref{rmk:weak-omega-functors} for the category $\WkCat{\omega}$ of weak $\omega$-categories and \emph{weak $\omega$-functors}.)

\begin{remark}
    Using the fact that $T$ is a cartesian monad, one can see that the entire structure of a cartesian monad over $T$ with contraction $(P,\eta^P,\mu^P,\ar^{P},\kappa^P)$ is determined by the tuple $(P1,\eta^P_1,\mu^P_1,\ar^{P}_1,\kappa^P)$ of its components at the terminal globular set 1. 
    This latter description leads to the notion of $T$-operad with contraction.
\end{remark}

\subsection{The free weak \texorpdfstring{$\omega$}{omega}-category monad \texorpdfstring{$L$}{L}}
For later reference, we shall describe the structure of the monad $L$ in some detail. 
From now on, we write $\ar^L_X$ as $\ar_X$, $\ar^L_1$ as $\ar$, and $\kappa^L$ as $\kappa$.

Given a globular set $X$, the globular set $LX$ is obtained as the pullback
\begin{equation*}
\begin{tikzpicture}[baseline=-\the\dimexpr\fontdimen22\textfont2\relax ]
      \node(00) at (-0.5,1) {$LX$};
      \node(01) at (1.5,1) {$L1$};
      \node(10) at (-0.5,-1) {$TX$};
      \node(11) at (1.5,-1) {$T1$.};
      
      \draw [->] (00) to node[auto, labelsize] {$L!$} (01); 
      \draw [->] (01) to node[auto, labelsize] {$\ar$} (11); 
      \draw [->] (00) to node[auto,swap,labelsize] {$\ar_X$} (10); 
      \draw [->] (10) to node[auto,swap,labelsize] {$T!$} (11); 
\end{tikzpicture}
\end{equation*}
Thus an $n$-cell of $LX$ is a pair $(\phi,\uu)$ consisting of $\phi\in (L1)_n$ and $\uu\in (TX)_n$ such that $\ar(\phi)=(T!)(\uu)$ (that is, $\uu$ is a pasting diagram in $X$ of shape $\ar(\phi)$; we call $\ar(\phi)$ the \emph{arity} of $\phi$).
We have $s^{LX}_m(\phi,\uu)=\bigl(s^{L1}_m(\phi),s^{TX}_m(\uu)\bigr)$ and $t^{LX}_m(\phi,\uu)=\bigl(t^{L1}_m(\phi),t^{TX}_m(\uu)\bigr)$ for each $0\leq m < n$.

Given a morphism $f\colon X\to Y$ of globular sets, the morphism $Lf\colon LX\to LY$ maps an $n$-cell $(\phi,\uu)$ of $LX$ to the $n$-cell $\bigl(\phi, (Tf)(\uu)\bigr)$ of $LY$.

For each $n\geq 0$, we denote the image of the unique $n$-cell of $1$ under the unit $\eta^L_1\colon 1\to L1$ of $L$ by $\widetilde{e}_n\in(L1)_n$.
Note that $\ar(\widetilde{e}_n)=[n]$.
For any globular set $X$, $\eta^L_X\colon X\to LX$ maps each $n$-cell $x\in X_n$ to the $n$-cell $(\widetilde{e}_n,[x])$ of $LX$.

Next we describe the multiplication $\mu^L$.
An $n$-cell of $L^21$ is a tuple $(\phi,\cchi)$, where $\phi\in (L1)_n$ and $\cchi\in(TL1)_n$ is a pasting diagram in $L1$ of shape $\ar(\phi)$.
$\mu^L_1\colon L^21\to L1$ maps $(\phi,\cchi)\in (L^21)_n$ to some $n$-cell  $\mu^L_1(\phi,\cchi)\in(L1)_n$.
For any globular set $X$, 
an $n$-cell of $L^2X$ is a tuple $(\phi,\widetilde{\uu})$, where $\phi\in(L1)_n$ and $\widetilde{\uu}\in(TLX)_n$ is a pasting diagram in $LX$ of shape $\ar(\phi)$. 
Using the above description of cells in $LX$, we may write $\widetilde{\uu}$ as  
\[
\widetilde{\uu}=
    \begin{bmatrix}
    (\chi_0,\uu_0) & &  \dots & & (\chi_r,\uu_r)\\
    & (\underline \chi_1,\underline \uu_1) & \dots & (\underline \chi_r,\underline \uu_r) &
    \end{bmatrix}.
\]
Thus $\widetilde{\uu}$ can be decomposed into
\[
\cchi=
    \begin{bmatrix}
    \chi_0 & &  \dots & & \chi_r\\
    & \underline \chi_1 & \dots & \underline \chi_r &
    \end{bmatrix}
    \in(TL1)_n
\]
and 
\[
\overline{\uu}=
    \begin{bmatrix}
    \uu_0 & & \dots & & \uu_r\\
    & \underline \uu_1 & \dots & \underline \uu_r &
    \end{bmatrix}
    \in(T^2X)_n.
\]
The morphism $\mu^L_X\colon L^2X\to LX$ maps $(\phi,\widetilde{\uu})\in(L^2X)_n$ as above to $\bigl(\mu^L_1(\phi,\cchi),\mu^T_X(\overline{\uu})\bigr)\in (LX)_n$.
(In fact, the above description applies to an arbitrary cartesian monad $P$ over $T$ in place of $L$.)

We now use the contraction~$\kappa$.
For any $n\geq 1$, parallel pair of $(n-1)$-cells $\phi,\phi'\in(L1)_{n-1}$ (inducing the morphism $\langle\phi,\phi'\rangle\colon \partial G^n\to L1$), and $n$-cell $\kk\in (T1)_n$ such that $\ar(\phi)=s^{T1}_{n-1}(\kk)=t^{T1}_{n-1}(\kk)=\ar(\phi')$, we have an $n$-cell $\kappa(\langle\phi,\phi'\rangle,\kk)\in (L1)_n$ such that $\kappa(\langle\phi,\phi'\rangle,\kk)\colon \phi\to  \phi'$ and $\ar\bigl(\kappa(\langle\phi,\phi'\rangle,\kk)\bigr)=\kk$:
\begin{equation*}
\begin{tikzpicture}[baseline=-\the\dimexpr\fontdimen22\textfont2\relax ]
      \node(00) at (0,1) {$\partial\OO^n$};
      \node(01) at (2,1) {$L1$};
      \node(10) at (0,-1) {$\OO^n$};
      \node(11) at (2,-1) {$T1$.};
      
      \draw [->] (00) to node[auto, labelsize] {$\ppair{\phi,\phi'}$} (01); 
      \draw [->] (01) to node[auto, labelsize] {$\ar$} (11); 
      \draw [->] (00) to node[auto,swap,labelsize] {$\iota_n$} (10); 
      \draw [->] (10) to node[auto,swap,labelsize] {$\kk$} (11);   
      \draw [->] (10) to node[midway,fill=white,labelsize] {$\kappa(\ppair{\phi,\phi'},\kk)$} (01);
\end{tikzpicture}
\end{equation*}

\begin{definition}
    For any $n\geq 0$ and $\kk\in(T1)_n$, we define $\spi(\kk)\in(L1)_n$ with $\ar(\spi(\kk))=\kk$ inductively (on $n$) as follows. 
    \begin{itemize}
        \item If $\kk=[n]\in (T1)_n$, then $\spi([n])=\widetilde{e}_n$.
        \item Otherwise, let $\kk\colon \kkd\to \kkd$, where $\kkd\in(T1)_{n-1}$. (Note that necessarily $n\geq 1$.) We define $\spi(\kk)=\kappa\bigl(\ppair{\spi(\kkd),\spi(\kkd)},\kk\bigr)$.
    \end{itemize}
    $\spi(\kk)$ is called the \emph{standard pasting instruction} of shape $\kk$. 
    $\spi\colon T1\to L1$ is a globular map which is a section of $\ar$, commuting with $\eta^T_1$ and $\eta^L_1$.
\end{definition}

Note that for any weak $\omega$-category $(X,\xi)$, any $n\geq 0$, and any $n$-dimensional pasting scheme $\kk\in (T1)_n$, we have the \emph{standard pasting operation of arity $\kk$} in $X$, mapping each pasting diagram $\uu\in (TX)_n$ in $X$ of shape $\kk$ to the $n$-cell 
$\xi\bigl(\spi(\kk),\uu\bigr)$ of $X$.
Moreover, any strict $\omega$-functor $f\colon (X,\xi)\to (Y,\nu)$ between weak $\omega$-categories preserves the standard pasting operations: for any pasting diagram $\uu$ in $X$ of shape $\kk$, we have 
\begin{align*}
        f\bigl(\xi\bigl(\spi(\kk),\uu\bigr)\bigr)
        = \nu\bigl((Lf)\bigl(\spi(\kk),\uu\bigr)\bigr)
        = \nu\bigl(\spi(\kk),(Tf)(\uu)\bigr).
\end{align*}

We introduce a notation (extending the corresponding notation for strict $\omega$-categories) for certain standard pasting operations we shall use in the sequel.

\begin{definition}
\label{def:id-and-comp-in-weak-omega-cat}
    Let $(X,\xi)$ be a weak $\omega$-category.
    \begin{enumerate}
        \item Given a natural number $n\geq 1$ and an $(n-1)$-cell $x$ of $X$, we define the $n$-cell 
    \[
    \id{n}{X}{x}=
    \xi\bigl( \spi([n-1]^{(n)}),[x]\bigr)
    \]
    of $X$.
        \item Given a natural number $n\geq 1$ and $n$-cells $u,v$ of $X$ such that $t_{n-1}^X(u)=s_{n-1}^X(v)$, we define the $n$-cell 
            \[
    u\comp{n-1}{X}v=
    \xi\biggl( \spi\biggl({\begin{bmatrix}
    n & & n \\
    & n-1 &
    \end{bmatrix}^{(n)}}\biggr),    \begin{bmatrix}
    u & & v \\
    & t^X_{n-1}(u)=s^X_{n-1}(v) &
    \end{bmatrix}\biggr)
    \]
    of $X$.
    \end{enumerate}
    When $X$ is clear from the context, we omit the superscript.
\end{definition}

These operations satisfy the following source and target formulas.

\begin{proposition}
\label{prop:s-and-t-for-id-and-ast}
    Let $(X,\xi)$ be a weak $\omega$-category.
    \begin{enumerate}
        \item Let $n\geq 1$ be a natural number and $x$ an $(n-1)$-cell of $X$.  Then we have 
             \[
             s_{n-1}(\id{n}{}{x})=x=t_{n-1}(\id{n}{}{x}).
             \]
        \item Let $n\geq 1$ be a natural number and $u,v$ $n$-cells of $X$ such that $t_{n-1}(u)=s_{n-1}(v)$ holds. Then we have 
        \[
             s_{n-1}(u\comp{n-1}{}v)=s_{n-1}(u)\quad\text{and}\quad
             t_{n-1}(u\comp{n-1}{}v)=t_{n-1}(v).
             \]
    \end{enumerate}
\end{proposition}
\begin{proof}
These are all straightforward consequences of the description of the source and target operations of $TX$ in \cref{subsec:T}. For example, the first equation in (2) can be proved as follows:
\begin{align*}
s_{n-1}^X(u\comp{n-1}{X}v) &= s_{n-1}^X\biggl(
    \xi\biggl( \spi\biggl({\begin{bmatrix}
    n & & n \\
    & n-1 &
    \end{bmatrix}^{(n)}}\biggr),    \begin{bmatrix}
    u & & v \\
    & t_{n-1}^X(u)=s_{n-1}^X(v) &
    \end{bmatrix}\biggr)\biggr)\\
    &= 
    \xi\biggl( s_{n-1}^{L1}\biggl(\spi\biggl({\begin{bmatrix}
    n & & n \\
    & n-1 &
    \end{bmatrix}^{(n)}}\biggr)\biggr), s^{TX}_{n-1}\biggl(   \begin{bmatrix}
    u & & v \\
    & t_{n-1}^X(u)=s_{n-1}^X(v) &
    \end{bmatrix}\biggr)\biggr)\\
    &= 
    \xi\biggl( \spi\biggl({\begin{bmatrix}
    n-1 \\
    \ 
    \end{bmatrix}^{(n-1)}}\biggr),    \begin{bmatrix}
    s_{n-1}^X(u) \\
    \ 
    \end{bmatrix}\biggr)\\
    & = \xi\circ \eta^L_X\bigl(s^X_{n-1}(u)\bigr)\\
    & = s_{n-1}^X(u).\qedhere
\end{align*}
\end{proof}

\begin{remark}
\label{rmk:composition-op-general}
    We generalise \cref{def:id-and-comp-in-weak-omega-cat} and \cref{prop:s-and-t-for-id-and-ast} in \cite{Fujii-Hoshino-Maehara-2-out-of-3}  because we need to use more general compositions there.
\end{remark}

Note that the monad law $\mu^L\circ L\eta^L=1_L$ allows us to write $\spi([n-1]^{(n)})\in(L1)_n$ as $\id{n}{L1}{\widetilde{e}_{n-1}}$ and 
\[
\spi\biggl({\begin{bmatrix}
    n & & n \\
    & n-1 &
    \end{bmatrix}^{(n)}}\biggr)\in(L1)_n
\]
as $\widetilde{e}_{n}\comp{n-1}{L1}\widetilde{e}_{n}$ for each $n\ge 1$, where $L1=(L1,\mu^L_1)$ is the free weak $\omega$-category on the terminal globular set $1$. In the following, we shall mainly use these latter notations.

For later reference, we remark on the compatibility of standard pasting operations and the construction of the \emph{hom weak $\omega$-category} $X(x,y)$ of a weak $\omega$-category $X$ between objects $x,y\in X_0$ defined in \cite[Section~9.3]{Leinster_book} and \cite{Cottrell_Fujii_hom}.
The latter is captured by the forgetful functor 
\[
    \Us\colon
    \WkCats{\omega}\to\enGph{(\WkCats{\omega})}
\]
mapping each weak $\omega$-category $X$ to the $(\WkCats{\omega})$-enriched graph \cite{Wolff_V-graph} $\Us X$ consisting of the same objects as $X$ together with the hom weak $\omega$-categories of $X$. In order to describe the definition of $\Us$, let us first observe that there is a sequence of maps $\Sigma=\bigl(\Sigma_n\colon (T1)_n\to (T1)_{n+1}\bigr)_{n\in\mathbb{N}}$ mapping 
    \[
    \kk = \begin{bmatrix}
    k_0 & & \dots & & k_r\\
    & \underline k_1 & \dots & \underline k_r &
    \end{bmatrix}
    \]
    to 
    \[
    \Sigma_n(\kk) = \begin{bmatrix}
    k_0+1 & & \dots & & k_r+1\\
    & \underline k_1+1 & \dots & \underline k_r+1 &
    \end{bmatrix}.
    \]
We also have a sequence of maps $\widetilde{\Sigma}=\bigl(\widetilde{\Sigma}_n\colon (L1)_n\to (L1)_{n+1}\bigr)_{n\in\mathbb{N}}$ which is compatible with the structure of $L$ and $\Sigma$; see \cite[Section~3]{Cottrell_Fujii_hom} for details. 
For any globular set $X$ and objects $x,y\in X_0$, we define the globular set $X(x,y)$ by 
\[
\bigl(X(x,y)\bigr)_n = \bigl\{\,u \in X_{n+1}\,\big\vert\,s^X_0(u)=x,\ \ t^X_0(u)=y\,\bigr\}
\]
with the evident source and target maps. If $X$ has a weak $\omega$-category structure $\xi\colon LX\to X$, then so does $X(x,y)$, with the structure map $\xi_{x,y}\colon L\bigl(X(x,y)\bigr)\to X(x,y)$ given by $\xi_{x,y}(\phi,\uu)=\xi\bigl(\widetilde{\Sigma}_n(\phi),\Sigma^{X,x,y}_n(\uu)\bigr)$, where $\Sigma^{X,x,y}_n\colon \bigl(T\bigl(X(x,y)\bigr)\bigr)_n\to (TX)_{n+1}$ is the inclusion.
\begin{proposition}
    $\widetilde{\Sigma}$ preserves standard pasting instructions.
\end{proposition}
\begin{proof}
    Among the compatibility of $\widetilde{\Sigma}$ with the structure of $L$ and $\Sigma$ are the properties that we have  $\widetilde{\Sigma}_n(\widetilde{e}_n)=\widetilde{e}_{n+1}$ for all $n\geq 0$, and that we have 
    \[\widetilde{\Sigma}_n\bigl(\kappa(\ppair{\phi,\phi'},\kk)\bigr)=\kappa\bigl(\ppair{\widetilde{\Sigma}_{n-1}(\phi),\widetilde{\Sigma}_{n-1}(\phi')},\Sigma_n(\kk)\bigr)\]
    for all $n\geq 1$, all parallel $\phi,\phi'\in (L1)_{n-1}$, and all $\kk\in(T1)_n$ with $\ar(\phi)=\ar(\phi')=s^{T1}_{n-1}(\kk)$.
    The claim follows from a straightforward induction on $n$.
\end{proof}

It follows that the standard pasting operation of arity $\kk\in(T1)_n$ on a hom weak $\omega$-category $X(x,y)$ is given by (a restriction of) the standard pasting operation of arity $\Sigma_n(\kk)\in(T1)_{n+1}$ on $X$. 
In particular, we shall use the following.

\begin{corollary}
\label{cor:composition-in-hom}
    Let $X$ be a weak $\omega$-category and $x,y\in X_0$. 
    \begin{enumerate}
        \item Given a natural number $n\geq 1$ and an $(n-1)$-cell $z$ of $X(x,y)$, we have
        \[
        \id{n}{X(x,y)}{z}=\id{n+1}{X}{z}.
        \]
        \item  Given a natural number $n\geq 1$ and $n$-cells $u,v$ of $X(x,y)$ such that $t_{n-1}^{X(x,y)}(u)=s_{n-1}^{X(x,y)}(v)$, we have
        \[
        u\comp{n-1}{X(x,y)}v=u\comp{n}{X}v.
        \]
    \end{enumerate}
\end{corollary}

\section{The main theorem}
\label{sec:invertible}
\label{sec:main}
In this section, we define (weakly) invertible cells in a weak $\omega$-category, and prove that they are closed under globular pasting operations.

\subsection{The definition of invertible cells in a weak \texorpdfstring{$\omega$}{omega}-category}
In order to define (weakly) invertible cells in a weak $\omega$-category $X$, we only need a small part of its structure.
Specifically, we use the following operations, where $n\geq 1$: the operation mapping each $(n-1)$-cell $x$ of $X$ to the $n$-cell $\id{n}{}{x}$ of $X$, and the operation mapping each pair of $n$-cells $u$ and $v$ of $X$ such that $t_{n-1}(u)=s_{n-1}(v)$ to the $n$-cell $u\comp{n-1}{}v$ of $X$; see \cref{def:id-and-comp-in-weak-omega-cat}. (Globular sets equipped with such operations are called \emph{$\omega$-precategories} in \cite{Cheng_dual}.)

\begin{definition}[{\cite{Cheng_dual,Lafont_Metayer_Worytkiewicz_folk_model_str_omega_cat}}]
\label{def:invertible}
    An $n$-cell $u \colon x \to y$ (with $n \ge 1$) in a weak $\omega$-category $X$ is \emph{weakly invertible} if there exist
    \begin{itemize}
        \item an $n$-cell $v \colon  y \to x$,
        \item a weakly invertible $(n+1)$-cell $p\colon u\comp{n-1}{}v\to \id{n}{}{x}$, and
        \item a weakly invertible $(n+1)$-cell $q\colon v\comp{n-1}{}u\to \id{n}{}{y}$
    \end{itemize}
    in $X$.
    In this situation, we say that $v$ is a \emph{pseudo inverse} of $u$.
    For $n$-cells $x$ and $y$ (with $n\geq 0$), we write $x \sim y$ if there exists a weakly invertible $(n+1)$-cell $u\colon x \to y$.
\end{definition}

Since in this paper the notion of \emph{(strictly) invertible cell} in a strict $\omega$-category seldom appears, we shall abbreviate ``weakly invertible'' to ``invertible'' and ``pseudo inverse'' to ``inverse'' in what follows.

\begin{remark}
\label{rmk:coinduction}
In \cref{def:invertible}, the notion of invertible cell in a weak $\omega$-category $X$ is defined \emph{coinductively}. 
We explain what this means in detail.

In general, if $\Psi\colon L\to L$ is a monotone map on a complete lattice $L$, then the set $\mathrm{Post}(\Psi)=\{\,s\in L\mid s\leq \Psi(s)\,\}$ of all post-fixed points of $\Psi$ is closed under arbitrary joins in $L$, and hence is also a complete lattice. In particular, $\mathrm{Post}(\Psi)$ has a largest element $t$. Since $\Psi(t)\in \mathrm{Post}(\Psi)$, we have in fact $t=\Psi(t)$, i.e., $t$ is a fixed point of $\Psi$. Thus \emph{any monotone map on a complete lattice has the greatest fixed point, which is also the greatest post-fixed point} \cite{Tarski_fixed_point}.

Now let us consider the (monotone) map $\Phi^X\colon \mathcal{P}(\coprod_{n\in \mathbb{N}}X_n)\to \mathcal{P}(\coprod_{n\in \mathbb{N}}X_n)$ on the powerset lattice of the set of all cells in a weak $\omega$-category $X$, mapping $S\subseteq \coprod_{n\in \mathbb{N}}X_n$ to the set of cells that are invertible up to $S$, or more precisely
\begin{multline*}
\Phi^X(S)=\bigl\{\,(u\colon x\to y)\in X_n\,\big\vert\, n\geq 1,\ \ \exists (v\colon y\to x)\in X_n,\\
\exists \bigl({p}\colon u\comp{n-1}{}v\to \id{n}{}{x}\bigr)\in S\cap X_{n+1},\ \ \exists \bigl({q}\colon v\comp{n-1}{}u\to \id{n}{}{y}\bigr)\in S\cap X_{n+1}\,\bigr\}.
\end{multline*}
Since $\mathcal{P}(\coprod_{n\in \mathbb{N}}X_n)$ is a complete lattice, $\Phi^X$ has a greatest (post-)fixed point $I$.
\cref{def:invertible} says that we define this $I$ to be the set of all invertible cells of $X$.
Observe that the characterisation of $I$ as the greatest post-fixed point of $\Phi^X$ yields the following proof method to show that a certain cell in $X$ is invertible: in order to show that a cell $u$ in $X$ is invertible, it suffices to find a set $W\subseteq \coprod_{n\in \mathbb{N}}X_n$ with $W\subseteq \Phi^X(W)$ and $u\in W$.
\end{remark}

\begin{remark}
\label{rmk:hom-weak-omega-cat}
    Let $X$ be a weak $\omega$-category.
    By \cref{cor:composition-in-hom}, for $n\geq 2$, an $n$-cell $u$ of $X$ is invertible if and only if it is invertible as an $(n-1)$-cell in the hom weak $\omega$-category $X\bigl(s_0^X(u),t_0^X(u)\bigr)$. (This is also commented in \cite[Section~6]{Cottrell_Fujii_hom}.)
\end{remark}


\subsection{Coherence}
We now prove elementary properties of invertible cells, and establish the existence of enough \emph{coherence} cells in a weak $\omega$-category, which plays an important role throughout this paper.
\begin{proposition}[{Cf.~\cite[Lemma~4.3]{Lafont_Metayer_Worytkiewicz_folk_model_str_omega_cat}}]
\label{lem:morphism-preserves-invertible-cells}
Any strict $\omega$-functor preserves invertible cells.
\end{proposition}
\begin{proof}
Let $f\colon X\to Y$ be a strict $\omega$-functor between weak $\omega$-categories. Define the set $W\subseteq \coprod_{n\geq 1}Y_n$ by
\[
W=\{\,u'\in Y_n\mid \text{$n\geq 1$ and there exists an invertible $n$-cell $u$ of $X$ with $u'=fu$}\,\}.
\]
We show that $W\subseteq \Phi^Y(W)$. If $u'\in W\cap Y_n$, then there exists an invertible $n$-cell $u\colon x\to x'$ of $X$ with $u'=fu$. So there exist an $n$-cell $v\colon x'\to x$ and invertible $(n+1)$-cells ${p}\colon u\comp{n-1}{X}v\to \id{n}{X}{x}$ and ${q}\colon v\comp{n-1}{X}u\to \id{n}{X}{x'}$ in $X$. Then $f{p},f{q}\in W$ and we have $f{p}\colon  fu\comp{n-1}{Y}fv\to\id{n}{Y}{fx}$ and $f{q}\colon fv\comp{n-1}{Y}fu\to \id{n}{Y}{fx'}$, showing $u'=fu\in \Phi^Y(W)$.
\end{proof}

A strict $\omega$-functor is \emph{contractible} (or \emph{locally a trivial fibration}) \cite[Section~9.1]{Leinster_book} if its underlying morphism of globular sets has the right lifting property with respect to $\iota_n\colon\partial\OO^n\to \OO^n$ for each $n\geq 1$. 

\begin{proposition}[{Cf.~\cite[Lemma~4.9]{Lafont_Metayer_Worytkiewicz_folk_model_str_omega_cat}}]
\label{prop:contractible-reflect-invertible}
    Any contractible strict $\omega$-functor reflects invertible cells.
\end{proposition}
\begin{proof}
    Let ${f}\colon X\to Y$ be a contractible strict $\omega$-functor between weak $\omega$-categories. Define the set $W\subseteq \coprod_{n\geq 1}X_n$ by 
    \[
    W=\{\,u\in X_n\mid n\geq 1\text{ and }{f}u\text{ is invertible}\,\}.
    \]
    We show that $W\subseteq \Phi^X(W)$. If $(u\colon x\to x')\in W\cap X_n$, then ${f}u\colon {f}x\to {f}x'$ is invertible in $Y$ and hence we can take its inverse $v\colon {f}x'\to {f}x$. By contractibility of ${f}$, there exists $\overline{v}\colon x'\to x$ in $X$ with ${f}\overline{v}=v$.
    We also have invertible cells ${p}\colon {f}u\comp{n-1}{Y}v\to \id{n}{Y}{{f}x}$ and ${q}\colon v\comp{n-1}{Y}{f}u\to \id{n}{Y}{{f}x'}$.
    Since ${f}u\comp{n-1}{Y}v={f}u\comp{n-1}{Y}{f}\overline{v}={f}(u\comp{n-1}{X}\overline{v})$ and $\id{n}{Y}{{f}x}={f}(\id{n}{X}{x})$, by contractibility of ${f}$ we obtain $\overline{{p}}\colon u\comp{n-1}{X}v\to \id{n}{X}{x}$ in $X$ with ${f}\overline{{p}}={p}$. Similarly, we obtain $\overline{{q}}\colon \overline{v}\comp{n-1}{X}u\to \id{n}{X}{x'}$ in $X$ with ${f}\overline{{q}}={q}$. Since $\overline{{p}},\overline{{q}}\in W$, we have $u\in\Phi^X(W)$.
\end{proof}

Recall that any strict $\omega$-category $(X,\gamma\colon TX\to X)$ can be regarded as a weak $\omega$-category $(X,\gamma\circ\ar_X\colon LX\to X)$.
\begin{proposition}[{\cite[Proposition~4.4 (1)]{Lafont_Metayer_Worytkiewicz_folk_model_str_omega_cat}}]
\label{prop:id-str-omega-cat-invertible}
    Any identity cell in a strict $\omega$-category is invertible.
\end{proposition}
\begin{proof}
Given a strict $\omega$-category $X$, consider the set $W$ of all identity cells in $X$.
\end{proof}

It follows that any strictly invertible cell in a strict $\omega$-category is invertible (in the sense of \cref{def:invertible}).
Note that in the strict $\omega$-category $(T1,\mu^T_1\colon T^21\to T1)$, an $n$-cell $\kk$ is an identity $n$-cell if and only if it is degenerate as an $n$-cell of $T1$. 
Similarly, for any globular set $X$, an $n$-cell $\uu$ of the strict $\omega$-category $(TX,\mu^T_X\colon T^2X\to TX)$ is an identity $n$-cell if and only if its shape is degenerate.

\begin{proposition}
\label{prop:base-case}
Let $(X,\xi)$ be a weak $\omega$-category, $n\geq 1$, and $(\phi,\uu)\in (LX)_n$.
If $\ar(\phi)\in (T1)_n$ is degenerate, then the $n$-cell $\xi(\phi,\uu)$ in $X$ is invertible.
\end{proposition}
\begin{proof}
Since $\xi$ is a strict $\omega$-functor from $(LX,\mu^L_X\colon L^2X\to LX)$ to $(X,\xi\colon LX\to X)$, by \cref{lem:morphism-preserves-invertible-cells} it suffices to show that the $n$-cell $(\phi,\uu)$ in $LX$ is invertible.
This follows from the facts that $\ar_X\colon LX\to TX$ is a contractible morphism of weak $\omega$-categories from $(LX,\mu^L_X\colon L^2X\to LX)$ to $(TX,\mu^T_X\circ\ar_{TX}\colon LTX\to TX)$, and that $\uu$ is an identity $n$-cell in the strict $\omega$-category $(TX,\mu^T_X\colon T^2X\to TX)$, by \cref{prop:id-str-omega-cat-invertible,prop:contractible-reflect-invertible}.
\end{proof}

We can now generalise \cref{prop:id-str-omega-cat-invertible} to weak $\omega$-categories. Namely, if $(X,\xi)$ is a weak $\omega$-category, $n\geq 0$, and $x\in X_n$, then $\id{n+1}{X}{x}\in X_{n+1}$ is invertible. This is because we have 
\[
\id{n+1}{X}{x}=\xi\bigl(\id{n+1}{L1}{\widetilde e_n},[x]\bigr)
\]
and $\ar\bigl(\id{n+1}{L1}{\widetilde e_n}\bigr)=\ar\bigl(\spi([n]^{(n+1)})\bigr)=[n]^{(n+1)}\in (T1)_{n+1}$ is degenerate.

More generally, notice that if $n\geq 0$ and $\phi,\phi'\in (L1)_n$ are parallel $n$-cells with $\ar(\phi)=\ar(\phi')=\kk\in (T1)_n$, then we can regard $\kk$ as an $(n+1)$-cell $\kk^{(n+1)}\colon \kk^{(n)}\to \kk^{(n)}$ of $T1$ and hence obtain an $(n+1)$-cell
\[
\kappa(\ppair{\phi,\phi'},\kk^{(n+1)})\colon \phi\to \phi'
\]
of $L1$.

\begin{proposition}[Coherence]
\label{cor:coherence}
Let $n\geq 0$ and $\phi,\phi'\in (L1)_n$ be parallel $n$-cells with $\ar(\phi)=\ar(\phi')=\kk\in (T1)_n$.
Then for any weak $\omega$-category $(X,\xi)$ and any pasting diagram $\uu$ of shape $\kk$ in $X$, the $(n+1)$-cell 
\[
\xi\bigl(\kappa(\ppair{\phi,\phi'},\kk^{(n+1)}),\uu\bigr)\colon \xi(\phi,\uu)\to\xi(\phi',\uu)
\]
in $X$ is invertible. 
In particular, we have $\xi(\phi,\uu)\sim\xi(\phi',\uu)$ in $X$.
\end{proposition}
\begin{proof}
Because
\[
\ar\bigl(\kappa(\ppair{\phi,\phi'},\kk^{(n+1)})\bigr)=\kk^{(n+1)}\colon \kk^{(n)}\to \kk^{(n)}
\]
is degenerate.
\end{proof}

Note that \cref{cor:coherence} shows that any two parallel (in the sense of being induced from parallel $n$-cells of $L1$) composites of a pasting diagram in a weak $\omega$-category are equivalent up to an invertible cell. 

\begin{remark}
\label{rmk:contractibility}
\cref{cor:coherence} is the uniqueness part of the pasting theorem promised in \cref{subsec:idea_of_weak_omega_cat}; it states precisely that, in a weak $\omega$-category, any two ways of composing a given $n$-dimensional pasting diagram yield the same composite up to invertible $(n+1)$-cell, as long as they agree on how to compose the boundary.
However, it is possible to give a different formulation of the uniqueness, namely as the contractibility of a suitable ``space'' of composites.

Recall that, given a weak $\omega$-category $X$ and $0$-cells $x,y \in X_0$, we may construct the {hom weak $\omega$-category} $X(x,y)$.
It is easy to see that if (the underlying globular map of) a strict $\omega$-functor $f\colon X \to Y$ has the right lifting property with respect to $\iota_k\colon \partial\OO^k \to \OO^k$ for some $k\geq 1$, then the induced map $f_{x,y}\colon X(x,y) \to Y(fx,fy)$ has the right lifting property with respect to $\iota_{k-1}\colon\partial\OO^{k-1} \to \OO^{k-1}$.

Now, let $X$ be a weak $\omega$-category, $n \ge 1$, $\kk$ an $n$-dimensional pasting scheme, $\uu$ a pasting diagram of shape $\kk$ in $X$, and $\phi,
\phi'\in (L1)_{n-1}$ parallel cells of arity $s^{TX}_{n-1}(\kk) = t^{TX}_{n-1}(\kk)$.
Then we can construct the following \emph{trivial fibration} (that is, a strict $\omega$-functor whose underlying globular map has the right lifting property with respect to $\iota_k\colon \partial\OO^k \to \OO^k$ for all $k \ge 0$) by repeatedly taking the hom weak $\omega$-categories:
\[
LX\bigl((\phi,s_{n-1}^{TX}(\uu)),(\phi',t_{n-1}^{TX}(\uu))\bigr) \to TX \bigl(s_{n-1}^{TX}(\uu), t_{n-1}^{TX}(\uu)\bigr).
\]
The codomain is a strict $\omega$-category, so there is a strict $\omega$-functor from the terminal weak $\omega$-category into it that picks out the $0$-cell $\uu$.
The pullback (in $\WkCats{\omega}$) of the resulting cospan is then the weak $\omega$-category of possible composites of $\uu$ satisfying the boundary conditions specified by $\phi$ and $\phi'$.
Since it admits a trivial fibration to the terminal weak $\omega$-category, one can reasonably call it a \emph{contractible space}; at least it is a weak $\omega$-groupoid (i.e., a weak $\omega$-category in which every cell of dimension $\geq 1$ is invertible) by \cref{prop:contractible-reflect-invertible,prop:id-str-omega-cat-invertible}.
\end{remark}

\subsection{The main theorem}

The main theorem of this paper is the following: if $(X,\xi)$ is a weak $\omega$-category and $(\phi,\uu)$ is an $n$-cell of $LX$ such that all $n$-cells of $X$ appearing in the pasting diagram $\uu$ are invertible in $X$, then the $n$-cell $\xi(\phi,\uu)$ is also invertible in $X$. 
Notice that the special case of this claim where $\uu$ does not contain any $n$-cell of $X$, is precisely \cref{prop:base-case}.
However, since the $n$-cell $(\phi,\uu)$ is not invertible in $LX$ whenever $\uu$ is non-degenerate, we will need more discussion in the general case.

\begin{definition}
\label{def:delta-q}
    Let $\kk$ be a pasting scheme of dimension $n$ and rank $r$.
    Let $0 \le i \le r$ and suppose $k_i = n$:
    \[
    \kk=
    \begin{bmatrix}
    k_0 & & \dots & & n & & \dots & & k_r\\
    & \underline k_1 & \dots & \underline k_{i} & & \underline k_{i+1} & \dots & \underline k_r &
    \end{bmatrix}.
    \]
    We write $\delta^i(\kk)$ for the $n$-dimensional pasting scheme defined by 
    \begin{align*}
    \delta^i(\kk)=\mu^T_1\left( 
    \begin{bmatrix}
    [k_0] & & \dots & & [n-1]^{(n)} & & \dots & & [k_r]\\
    & [\underline k_1] & \dots & [\underline k_{i}] & & [\underline k_{i+1}] & \dots & [\underline k_r] &
    \end{bmatrix}\right).
    \end{align*}
    Explicitly, we have the following description of $\delta^i(\kk)$.
    \begin{itemize}
         \item If $i > 0$ and $\underline k_i = n-1$, then $\delta^i(\kk)$ is obtained from $\kk$ by removing $k_i$ and $\underline k_i$. 
        \item If $i < r$ and $\underline k_{i+1} = n-1$, then $\delta^i(\kk)$ is obtained from $\kk$ by removing $k_i$ and $\underline k_{i+1}$.
        \item Otherwise, $\delta^i(\kk)$ is obtained from $\kk$ by replacing $k_i = n$ by $n-1$.
    \end{itemize}
    (Note that, although it is possible for $\kk$ to satisfy the premises of both the first and second clauses, the two definitions of $\delta^i(\kk)$ agree in that case.)
    
    Let $\uu$ be a pasting diagram of shape $\kk$ as above in a weak $\omega$-category $X$.
    Suppose that there exists $x \in X_{n-1}$ such that $u_i = \id{n}{X}{x}$:
    \[
    \uu=\begin{bmatrix}
    u_0 & & \dots & & \id{n}{X}{x} & & \dots & & u_r\\
    & \underline u_1 & \dots & \underline u_{i} & & \underline u_{i+1} & \dots & \underline u_r &
    \end{bmatrix}.
    \]
    We write $\delta^i(\uu)$ for the pasting diagram of shape $\delta^i(\kk)$ in $X$ defined by 
    \[
    \delta^i(\uu)=
    \mu^T_X\left( 
    \begin{bmatrix}
    [u_0] & & \dots & & [x]^{(n)} & & \dots & & [u_r]\\
    & [\underline u_1] & \dots & [\underline u_{i}] & & [\underline u_{i+1}] & \dots & [\underline u_r] &
    \end{bmatrix}\right).
    \]
    Explicitly, we have the following description of $\delta^i(\uu)$.
    \begin{itemize}
        \item If $i > 0$ and $\underline k_i = n-1$, then $\delta^i(\uu)$ is obtained from $\uu$ by removing $u_i$ and $\underline u_i = x$.
        \item If $i < r$ and $\underline k_{i+1} = n-1$, then $\delta^i(\uu)$ is obtained from $\uu$ by removing $u_i$ and $\underline u_{i+1} = x$.
        \item Otherwise, $\delta^i(\uu)$ is obtained from $\uu$ by replacing $u_i$ by $x$.
    \end{itemize}
        (Note that, although it is possible for $\uu$ to satisfy the premises of both the first and second clauses, the two definitions of $\delta^i(\uu)$ agree in that case.)

Let $\phi$ be an $n$-cell of $L1$ with $\ar(\phi)=\kk$ as above. We write $\delta^i(\phi)$ for the $n$-cell of $L1$ defined as $\delta^i(\phi)=\kappa\bigl(\ppair{s^{L1}_{n-1}(\phi),t^{L1}_{n-1}(\phi)},\delta^i(\kk)\bigr)$.
\end{definition}

\begin{example}
\label{example-of-delta}
Let 
\[
\kk=
\begin{bmatrix}
    2 & & 2 & & 2 & & 1 & & 2 \\
    & 1 & & 1 & & 0 & & 0 & 
\end{bmatrix},
\]
and suppose we have a $2$-dimensional pasting diagram 
\[
\uu=
\begin{bmatrix}
    \alpha & & \id{2}{}{g} & & \beta & & i & & \gamma  \\
    & g & & g & & b & & c & 
\end{bmatrix}
\]
of shape $\kk$ in a weak $\omega$-category $X$, which may be depicted as 
\begin{equation}
\label{eqn:pasting-diag-with-ix}
\begin{tikzpicture}[baseline=-\the\dimexpr\fontdimen22\textfont2\relax ]
      \node(0) at (0,0) {$a$};
      \node(1) at (2,0) {$b$};
      \node(2) at (4,0) {$c$};
      \node(3) at (6,0) {$d$.};

      \draw [->]  (0) .. controls (0,2) and (2,2) .. node[auto,labelsize] {$f$} (1);
      \draw [->,bend left=70]  (0) to node[labelsize,midway,fill=white] {$g$} (1);
      \draw [->,bend right=70] (0) to node[labelsize,midway,fill=white] {$g$} (1); 
      \draw [->]  (0) .. controls (0,-2) and (2,-2) .. node[auto,swap,labelsize] {$h$} (1);

      \draw [->] (1,1.4) to node[auto,labelsize] {$\alpha$} (1,0.9);
      \draw [->] (1,0.5) to node[auto,font=\tiny] {$\id{2}{}{g}$} (1,-0.5);
      \draw [->] (1,-0.9) to node[auto,labelsize] {$\beta$} (1,-1.4);

      \draw [->] (1) to node (20) {} (2);
      \path (1) to node[auto,labelsize] {$i$} (2);
      
      \draw [->,bend left=30]  (2) to node[auto,labelsize] {$j$} (3); 
      \draw [->,bend right=30]  (2) to node[auto,swap,labelsize] {$k$} (3);
      
      \draw [->] (5,0.25) to node[auto,labelsize] {$\gamma$} (5,-0.25);
\end{tikzpicture}
\end{equation}
Then we have
\[
\delta^1(\uu)=
\begin{bmatrix}
    \alpha & & \beta & & i & & \gamma \\
    & g & & b & & c & 
\end{bmatrix}
\]
which may be depicted as
\begin{equation}
\label{eqn:pasting-diag-without-ix}
\begin{tikzpicture}[baseline=-\the\dimexpr\fontdimen22\textfont2\relax ]
      \node(0) at (0,0) {$a$};
      \node(1) at (2,0) {$b$};
      \node(2) at (4,0) {$c$};
      \node(3) at (6,0) {$d$.};

      \draw [->,bend left=70]  (0) to node[auto,labelsize] {$f$} (1); 
      \draw [->]  (0) to node[labelsize,midway,fill=white] {$g$} (1);
      \draw [->,bend right=70]  (0) to node[auto,swap,labelsize] {$h$} (1);

      \draw [->] (1,0.6) to node[auto,labelsize] {$\alpha$} (1,0.2);
      \draw [->] (1,-0.2) to node[auto,labelsize] {$\beta$} (1,-0.6);

      \draw [->] (1) to node (20) {} (2);
      \path (1) to node[auto,labelsize] {$i$} (2);
      
      \draw [->,bend left=30]  (2) to node[auto,labelsize] {$j$} (3); 
      \draw [->,bend right=30]  (2) to node[auto,swap,labelsize] {$k$} (3);
      
      \draw [->] (5,0.25) to node[auto,labelsize] {$\gamma$} (5,-0.25);
\end{tikzpicture}\qedhere
\end{equation}
\end{example}

The following lemma formalises the idea that composites of diagrams such as  \cref{eqn:pasting-diag-with-ix} and \cref{eqn:pasting-diag-without-ix} in a weak $\omega$-category should be equivalent up to an invertible cell.

\begin{proposition}[Unit Law]
\label{lem:unit-law}
    Let $n\geq 1$, $\kk$ be a pasting scheme of dimension $n$ and rank $r$, $(X,\xi)$ a weak $\omega$-category, and $(\phi,\uu)\in (LX)_n$ an $n$-cell such that $\ar(\phi)=\kk$.
    Let $0 \le i \le r$ and suppose $k_i = n$ and $u_i = \id{n}{}{x}$ for some $x \in X_{n-1}$.
    Then we have 
    \[
    \xi(\phi,\uu) \sim \xi\bigl(\delta^i(\phi),\delta^i(\uu)\bigr)
    \]
    in $X$.
\end{proposition}
\begin{proof}
By the assumption, we have 
    \begin{align*}
    \kk 
    &= \begin{bmatrix}
    k_0 & & \dots & & n & & \dots & & k_r\\
    & \underline k_1 & \dots & \underline k_{i} & & \underline k_{i+1} & \dots & \underline k_r &
    \end{bmatrix}\in(T1)_n
    \end{align*}
and
    \[
    \uu = \begin{bmatrix}
    u_0 & & \dots & & \id{n}{}{x} & & \dots & & u_r\\
    & \underline u_1 & \dots & \underline u_{i} & & \underline u_{i+1} & \dots & \underline u_r &
    \end{bmatrix}\in(TX)_n.
    \]
Define a pasting diagram $\widetilde{\uu}$ in $LX$ by 
    \[
    \widetilde{\uu}
    =
    \begin{bmatrix}
    (\widetilde{e}_{k_0},[u_0]) & & \dots & & (\id{n}{L1}{\widetilde e_{n-1}},[x]) & & \dots & & (\widetilde{e}_{k_r},[u_r])\\
    & (\widetilde{e}_{\underline k_1},[\underline u_1]) & \dots & (\widetilde{e}_{\underline k_{i}},[\underline u_{i}]) & & (\widetilde{e}_{\underline k_{i+1}},[\underline u_{i+1}]) & \dots & (\widetilde{e}_{\underline k_r},[\underline u_r]) &
    \end{bmatrix}.
    \]
    Since $\widetilde{\uu}$ is also of shape $\kk$, we obtain $(\phi,\widetilde{\uu})\in (L^2X)_n$. By $\xi(\widetilde{e}_{k_j},[u_j])=u_j$, $\xi(\widetilde{e}_{\underline k_j},[\underline u_j])=\underline u_j$, and $\xi(\id{n}{L1}{\widetilde e_{n-1}},[x])=\id{n}{X}{x}$, we have 
    \begin{align*}
        (L\xi)(\phi,\widetilde{\uu})&=\bigl(\phi,(T\xi)(\widetilde{\uu})\bigr)=(\phi,\uu)\in(LX)_n.
    \end{align*} 
    
    Let us calculate $\mu^L_X(\phi,\widetilde{\uu})\in (LX)_n$. To this end, we decompose $\widetilde{\uu}\in(TLX)_n$ into 
    \[
    \cchi=
    \begin{bmatrix}
    \widetilde{e}_{k_0} & & \dots & & \id{n}{L1}{\widetilde e_{n-1}} & & \dots & & \widetilde{e}_{k_r}\\
    & \widetilde{e}_{\underline k_1} & \dots & \widetilde{e}_{\underline k_{i}} & & \widetilde{e}_{\underline k_{i+1}} & \dots & \widetilde{e}_{\underline k_r} &
    \end{bmatrix}\in (TL1)_n
    \]
    and 
    \[
    \overline{\uu}=
    \begin{bmatrix}
    [u_0] & & \dots & & [x]^{(n)} & & \dots & & [u_r]\\
    & [\underline u_1] & \dots & [\underline u_{i}] & & [\underline u_{i+1}] & \dots & [\underline u_r] &
    \end{bmatrix} \in (T^2X)_n.
    \]
    Then we have 
    \[
    \mu^L_X(\phi,\widetilde{\uu})=\bigl(\mu^L_1(\phi,\cchi),\mu^T_X(\overline{\uu})\bigr)=\bigl(\mu^L_1(\phi,\cchi),\delta^i(\uu)\bigr).
    \]

    On the other hand, the $n$-cells $\mu^L_1(\phi,\cchi)$ and $\delta^i(\phi)=\kappa\bigl(\ppair{s_{n-1}^{L1}(\phi),t_{n-1}^{L1}(\phi)},\delta^i(\kk)\bigr)$ of $L1$ are parallel.
    This is because 
    \begin{align*}
        s_{n-1}^{L1}\big(\mu^L_1(\phi,\cchi)\big)
        &=\mu^L_1\big(s_{n-1}^{L1}(\phi),s_{n-1}^{TL1}(\cchi)\big)\\
        &=\mu^L_1\Bigl(s_{n-1}^{L1}(\phi),(T\eta^L_1)\bigl(\ar(s_{n-1}^{L1}(\phi))\bigr)\Bigr)\\
        &=\mu^L_1\circ (L\eta^L_1)(s_{n-1}^{L1}(\phi))\\
        &=s_{n-1}^{L1}(\phi)
    \end{align*}
    and similarly $t_{n-1}^{L1}\big(\mu^L_1(\phi,\cchi)\big)=t_{n-1}^{L1}(\phi)$.
    Hence by \cref{cor:coherence}, we have 
    \begin{align*}
    \xi(\phi,\uu)
    &=
    \xi\circ (L\xi)(\phi,\widetilde{\uu})\\
    &=\xi\circ \mu^L_X(\phi,\widetilde{\uu})\\
    &=\xi\bigl(\mu^L_1(\phi,\cchi),\delta^i(\uu)\bigr)\\
    &\sim \xi\bigl(\delta^i(\phi),\delta^i(\uu)\bigr).\qedhere
    \end{align*}
\end{proof}

\begin{definition}
    Let $\kk$ be a pasting scheme of dimension $n$ and rank $r$, and let $\uu$ be a pasting diagram of shape $\kk$ in a globular set $X$.
    By the \emph{set of full-dimensional labels in $\uu$}, we mean the set
    \[
    \fulllabel(\uu) = \{\,u_i \mid 0 \le i \le r, \quad k_i = n\,\} \subseteq X_n.
    \]
    
    Let $S \subseteq X$ be a set of cells of a weak $\omega$-category $X$ (i.e., it consists of $S_n \subseteq X_n$ for each $n \ge 0$ subject to no conditions).
    By the \emph{set of $S$-labelled pastings}, we mean the set
    \[
    \pst(S) = \bigl\{\,\xi(\phi,\uu) \,\big\vert\, n\geq 0,\quad (\phi,\uu) \in (LX)_n, \quad \fulllabel(\uu) \subseteq S\,\bigr\}
    \]
    of cells of $X$.
\end{definition}

Note that the set $\pst(S)$ may contain $n$-cells in $S$ whiskered with $m$-cells not in $S$ with $m < n$.
Using $\pst$, we can now state our main theorem as follows. 
\begin{theorem}
\label{thm:main}
Let $(X,\xi)$ be a weak $\omega$-category and let $I \subseteq X$ be the set of all invertible cells in $X$.
Then we have $\pst(I)\subseteq I$.
\end{theorem}

We need some definitions for the proof of \cref{thm:main}.
For each $n\geq 0$ and each $n$-dimensional pasting scheme $\kk$ of rank $r$, define
\[
\|\kk\|^{(n)}=\bigl\lvert\bigl\{\,i\in\{\,0,1,\dots,r\,\}\,\big\vert\, k_i=n \,\bigr\}\bigr\rvert.
\]
Notice that $\kk$ is non-degenerate as an $n$-cell of $T1$ if and only if $\|\kk\|^{(n)}>0$.

\begin{definition}
    Let $X$ be a weak $\omega$-category and $S\subseteq X$ be a set of cells of $X$.
    Given any $n\geq 1$ and any $n$-cell $u\colon x\to y$ of $X$, an $n$-cell $v\colon y\to x$ of $X$ is called an \emph{$S$-inverse} of $u$ if there exist $(n+1)$-cells ${p}\colon u\comp{n-1}{}v\to \id{n}{}{x}$ and ${q}\colon v\comp{n-1}{}u\to\id{n}{}{y}$ in $S$.
    
    Let $n\geq 1$, $(\phi,\uu)\in(LX)_n$, and $\kk=\ar(\phi)$. An \emph{$S$-inverse instruction} of $(\phi,\uu)$ is an $n$-cell $(\phiinv,\uuinv)$ of $LX$ satisfying the following conditions.
    \begin{itemize}
        \item $s^{L1}_{n-1}(\phiinv)=t^{L1}_{n-1}(\phi)$, 
        \item $t^{L1}_{n-1}(\phiinv)=s^{L1}_{n-1}(\phi)$, and
        \item $\uuinv\in(TX)_n$ is obtained from $\uu$ by replacing, for each $(n-1)$-transversal component $0 \le i \le j \le r$ of $\kk$, the corresponding segment
    \[
    \begin{bmatrix}
    u_{i} & & \dots & & u_{j}\\
    & \underline u_{i+1} & \dots & \underline u_{j} &
    \end{bmatrix}
    \]
    with
    \[
    \begin{bmatrix}
    v_{j} & & \dots & & v_{i}\\
    & \underline u_{j} & \dots & \underline u_{i+1} &
    \end{bmatrix},
    \]
    where $v_{{l}}$ is an $S$-inverse of $u_{{l}}$ for each $i\le l\le j$.\qedhere
    \end{itemize}
\end{definition}

   Note that $u$ admits an $S$-inverse if and only if $u\in\Phi(S)$, and $(\phi,\uu)$ admits an $S$-inverse instruction if and only if $\fulllabel(\uu)\subseteq \Phi(S)$. Also note that the shape of $\uuinv$ is the same as that of $\uu$ (recall that the cell $\underline u_l$ is of dimension $n-1$ for each $i+1\leq l\leq j$), and that the types of $(\phi,\uu)$ and $(\phiinv,\uuinv)$ in $LX$ are related by 
\begin{equation}
\label{eqn:type-phi-u-phiinv-uinv}
\begin{tikzpicture}[baseline=-\the\dimexpr\fontdimen22\textfont2\relax ]
      \node(0) at (0,0.5) {$s_{n-1}^{LX}(\phi,\uu)$};
      \node(1) at (5,0.5) {$t_{n-1}^{LX}(\phi,\uu)$};
      \node(2) at (0,-0.5) {$t_{n-1}^{LX}(\phiinv,\uuinv)$};
      \node(3) at (5,-0.5) {$s_{n-1}^{LX}(\phiinv,\uuinv)$.};
      
      \draw [->] (0) to node[auto, labelsize] {$(\phi,\uu)$} (1); 
      \draw [->] (3) to node[auto, labelsize] {$(\phiinv,\uuinv)$} (2); 
    \node[rotate=90] at (0,0) {$=$};
    \node[rotate=90] at (5,0) {$=$};
\end{tikzpicture}
\end{equation}
(Here we are not claiming that this diagram ``commutes.'')

Although the proof of \cref{thm:main} is rather long, the underlying idea is simple; we show that any $n$-cell of the form $\xi(\phi,\uu)$ with $\fulllabel(\uu) \subseteq I$ admits an ($I$-)inverse, namely $\xi(\phiinv,\uuinv)$, where $(\phiinv,\uuinv)$ is an $I$-inverse instruction of $(\phi,\uu)$ (whose existence follows from $\fulllabel(\uu)\subseteq I=\Phi(I)$).
The non-trivial part is constructing invertible $(n+1)$-cells witnessing the invertibility of $\xi(\phi,\uu)$.
Even in the simplest case of $u_0\comp{0}{}u_1$, where the $1$-cells $u_0\colon x\to y$ and $u_1\colon y\to z$ admit inverses $v_0\colon y\to x$ and $v_1\colon z\to y$ respectively, connecting the composite
$(u_0\comp{0}{}u_1)\comp{0}{}(v_1\comp{0}{}v_0)$
to $\id{1}{}{x}$ requires
\begin{itemize}
    \item rebracketing the expression (using coherence) so that we are composing $u_1$ and $v_1$ first,
    \item whiskering with $u_0$ and $v_0$ the $2$-cell ${p}_1$ witnessing the invertibility of $u_1$,
    \item applying the unit law to obtain $u_0\comp{0}{}v_0$ (getting rid of an extra identity in the middle), and
    \item using the $2$-cell ${p}_0$ witnessing the invertiblity of $u_0$.
\end{itemize}
The following diagram illustrates the situation.
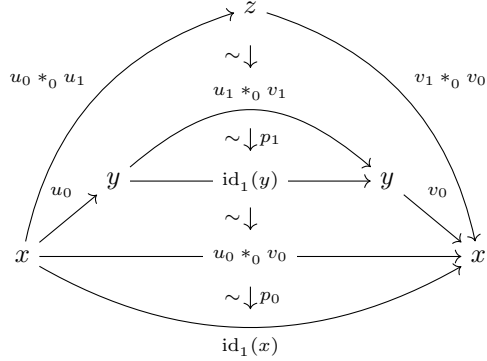
\[
\begin{tikzpicture}[baseline=-\the\dimexpr\fontdimen22\textfont2\relax ]
      \node(0) at (0,0) {$x$};
      \node(2) at (6,0) {$x$};
      \node(3) at (1.2,1) {$y$};
      \node(5) at (4.8,1) {$y$};
      \node(6) at (3,3.3) {$z$};

      \draw [->,bend left=30]  (0) to node[labelsize,auto] {$u_0\comp{0}{}u_1$} (6);
      \draw [->,bend left=30]  (6) to node[labelsize,auto] {$v_1\comp{0}{}v_0$} (2);
      \draw [->]  (0) to node[labelsize,midway, fill=white] {$u_0\comp{0}{}v_0$} (2);
      \draw [->]  (3) .. controls (2.6,2.2) and (3.4,2.2) .. node[labelsize,auto] {$u_1\comp{0}{}v_1$} (5);
     \draw [->]  (0) to node[labelsize,auto,near end] {$u_0$} (3);
     \draw [->]  (5) to node[labelsize,auto,near start] {$v_0$} (2);
      \draw [->,bend right=30] (0) to node[labelsize,auto,swap] {$\id{1}{}{x}$} (2); 
      \draw [->]  (3) to node[labelsize,midway,fill=white] {$\id{1}{}{y}$} (5);

      \draw [->] (3,2.8) to node[auto,swap,labelsize] {$\sim$} (3,2.5);
      \draw [->] (3,1.7) to node[auto,labelsize] {${p}_1$} node[auto,labelsize,swap] {$\sim$} (3,1.4);
      \draw [->] (3,0.65) to  node[auto,labelsize,swap] {$\sim$} (3,0.35);
      \draw [->] (3,-0.4) to node[auto,labelsize] {${p}_0$} node[auto,labelsize,swap] {$\sim$} (3,-0.7);
\end{tikzpicture}
\]
Moreover, we must show that the resulting $2$-cell is itself invertible.
In the actual proof, this last part is treated by considering certain (pre-)fixed points of $\pst$.

When connecting $(u_0 \comp{0}{} u_1) \comp{0}{} (v_1 \comp{0}{}v_0)$ to $\id{1}{}{x}$ in this example, we first cancelled out just one of the inverse pairs (namely $u_1$ and $v_1$) and obtained $u_0 \comp{0}{} v_0$.
In the general case, writing out such intermediate composites can be rather cumbersome.
This is our motivation for introducing the following notation, which allows us to write e.g.
\[
\delta^1_+ \left(\begin{bmatrix}
    u_0 & & u_1\\
    & y &
\end{bmatrix}\right) = \begin{bmatrix}
    u_0 \\ \
\end{bmatrix}
\quad\text{and}\quad
\delta^0_- \left(\begin{bmatrix}
    v_1 & & v_0\\
    & y &
\end{bmatrix}\right) = \begin{bmatrix}
    v_0 \\ \
\end{bmatrix}.
\]

\begin{definition}
    Let $\kk$ be a pasting scheme of dimension $n$ and rank $r$.
    Let $0 \le i \le r$ and suppose $k_i = n$:
    \[
    \kk=
    \begin{bmatrix}
    k_0 & & \dots & & n & & \dots & & k_r\\
    & \underline k_1 & \dots & \underline k_{i} & & \underline k_{i+1} & \dots & \underline k_r &
    \end{bmatrix}.
    \]
    Recall the pasting scheme $\delta^i(\kk)$ of \cref{def:delta-q}.
    Let $\uu$ be a pasting diagram of shape $\kk$ in a weak $\omega$-category $X$.
\begin{enumerate}
    \item Suppose that either $i=r$ or $\underline k_{i+1}<n-1$ holds. 
    We write $\delta^i_+(\uu)$ for the following pasting diagram of shape $\delta^i(\kk)$ in $X$.
    \begin{itemize}
        \item If $i > 0$ and $\underline k_i = n-1$, then $\delta^i_+(\uu)$ is obtained from $\uu$ by removing $u_i$ and $\underline u_i$.
        \item Otherwise, $\delta^i_+(\uu)$ is obtained from $\uu$ by replacing $u_i$ by $s^X_{n-1}(u_i)$.
    \end{itemize}
    \item Suppose that either $i=0$ or $\underline k_{i}<n-1$ holds. 
    We write $\delta^i_-(\uu)$ for the following pasting diagram of shape $\delta^i(\kk)$ in $X$.
    \begin{itemize}
        \item If $i < r$ and $\underline k_{i+1} = n-1$, then $\delta^i_-(\uu)$ is obtained from $\uu$ by removing $u_i$ and $\underline u_{i+1}$.
        \item Otherwise, $\delta^i_-(\uu)$ is obtained from $\uu$ by replacing $u_i$ by $t^X_{n-1}(u_i)$.\qedhere
    \end{itemize}
\end{enumerate}
\end{definition}

\begin{example}
    For\[
\kk=
\begin{bmatrix}
    2 & & 2 & & 2 & & 1 & & 2 \\
    & 1 & & 1 & & 0 & & 0 & 
\end{bmatrix}
\quad\text{and}\quad
\uu=
\begin{bmatrix}
    \alpha & & \id{2}{}{g} & & \beta & & i & & \gamma  \\
    & g & & g & & b & & c & 
\end{bmatrix}
\]
of \cref{example-of-delta}, we have
    \[
\delta^2_+(\uu)=
\begin{bmatrix}
    \alpha & & \id{2}{}{g} & & i & & \gamma \\
    & g & & b & & c & 
\end{bmatrix},
\]
which may be depicted as
\[
\begin{tikzpicture}[baseline=-\the\dimexpr\fontdimen22\textfont2\relax ]
      \node(0) at (0,0) {$a$};
      \node(1) at (2,0) {$b$};
      \node(2) at (4,0) {$c$};
      \node(3) at (6,0) {$d$,};

      \draw [->,bend left=70]  (0) to node[auto,labelsize] {$f$} (1); 
      \draw [->]  (0) to node[labelsize,midway,fill=white] {$g$} (1);
      \draw [->,bend right=70]  (0) to node[auto,swap,labelsize] {$g$} (1);

      \draw [->] (1,0.6) to node[auto,labelsize] {$\alpha$} (1,0.2);
      \draw [->] (1,-0.2) to node[auto,near start,font=\tiny] {$\id{2}{}{g}$} (1,-0.6);

      \draw [->] (1) to node (20) {} (2);
      \path (1) to node[auto,labelsize] {$i$} (2);
      
      \draw [->,bend left=30]  (2) to node[auto,labelsize] {$j$} (3); 
      \draw [->,bend right=30]  (2) to node[auto,swap,labelsize] {$k$} (3);
      
      \draw [->] (5,0.25) to node[auto,labelsize] {$\gamma$} (5,-0.25);
\end{tikzpicture}
\]
and
\[
\delta^4_-(\uu)=
\begin{bmatrix}
    \alpha & & \id{2}{}{g} & & \beta & & i & & k \\
    & g & & g & & b & & c & 
\end{bmatrix},
\]
which may be depicted as
\[
\begin{tikzpicture}[baseline=-\the\dimexpr\fontdimen22\textfont2\relax ]
      \node(0) at (0,0) {$a$};
      \node(1) at (2,0) {$b$};
      \node(2) at (4,0) {$c$};
      \node(3) at (6,0) {$d$.};

      \draw [->]  (0) .. controls (0,2) and (2,2) .. node[auto,labelsize] {$f$} (1);
      \draw [->,bend left=70]  (0) to node[labelsize,midway,fill=white] {$g$} (1);
      \draw [->,bend right=70] (0) to node[labelsize,midway,fill=white] {$g$} (1); 
      \draw [->]  (0) .. controls (0,-2) and (2,-2) .. node[auto,swap,labelsize] {$h$} (1);

      \draw [->] (1,1.4) to node[auto,labelsize] {$\alpha$} (1,0.9);
      \draw [->] (1,0.5) to node[auto,font=\tiny] {$\id{2}{}{g}$} (1,-0.5);
      \draw [->] (1,-0.9) to node[auto,labelsize] {$\beta$} (1,-1.4);

      \draw [->] (1) to node (20) {} (2);
      \path (1) to node[auto,labelsize] {$i$} (2);
      
      \draw [->]  (2) to node[auto,labelsize] {$k$} (3);
\end{tikzpicture}\qedhere
\]
\end{example}



\begin{proof}[Proof of \cref{thm:main}]
Let $\mathrm{Pre}(\pst)_I$ be the set of all pre-fixed points of the monotone map 
\[\pst\colon \mathcal{P}\Big(\coprod_{n\in\N}{X_n}\Big)\to \mathcal{P}\Big(\coprod_{n\in\N}{X_n}\Big)\] 
containing $I$, i.e.,
\[
\mathrm{Pre}(\pst)_I=\Bigl\{\,S\subseteq\coprod_{n\in\N}{X_n} \ \Big\vert \ I\subseteq S,\quad \pst(S)\subseteq S \,\Bigr\}.
\]
Since $\mathrm{Pre}(\pst)_I$ is closed under arbitrary intersections, it has the smallest element $J$. Note that since $S\subseteq \pst(S)$ holds for any $S\subseteq \coprod_{n\in\N}X_n$, all pre-fixed points of $\pst$ are in fact fixed points.

We shall show that \emph{$\mathrm{Pre}(\pst)_I$ is closed under $\Phi$}. Then in particular we have $\Phi(J)\in\mathrm{Pre}(\pst)_I$, and hence $J\subseteq \Phi(J)$ by the minimality of $J$. This implies $J\subseteq I$. On the other hand, $I\subseteq J$ holds since $J\in \mathrm{Pre}(\pst)_I$. Therefore we have $I=J\in\mathrm{Pre}(\pst)_I$ and in particular $\pst(I)\subseteq I$, as desired.

To this end, take $S\in \mathrm{Pre}(\pst)_I$. Then $\Phi(S)$ contains $I$ since $I=\Phi(I)\subseteq \Phi(S)$. Thus it suffices to show that $\pst(\Phi(S))\subseteq \Phi(S)$ holds.
Since every cell in $\pst(\Phi(S))$ is of the form $\xi(\phi,\uu)$ for some pasting instruction $(\phi,\uu)$ (of dimension $\ge 1$) that admits an $S$-inverse instruction, we would want to prove that:
\begin{equation}
\label{eqn:IH}
    \parbox{\dimexpr\linewidth-5em}{for each $n\geq 1$, each $(\phi,\uu)\in(LX)_n$,
        and each $S$-inverse instruction $(\phiinv,\uuinv)$ of $(\phi,\uu)$, there exist $(n+1)$-cells
\[
\xi(\phi,\uu)\comp{n-1}{X}\xi(\phiinv,\uuinv)\to \id{n}{X}{s^X_{n-1}\xi(\phi,\uu)}
\]
and 
\[
\xi(\phiinv,\uuinv)\comp{n-1}{X}\xi(\phi,\uu)\to \id{n}{X}{t^X_{n-1}\xi(\phi,\uu)}
\]
in $S$ (or equivalently in $\pst(S)$).}
\end{equation}
We shall do so by induction on $ \|{\ar(\phi)}\|^{(n)}$.

The base case $\|{\ar(\phi)}\|^{(n)}=0$ is covered by \cref{prop:base-case}; recall that $I\subseteq S$.
So fix $N\geq 0$ and suppose that \cref{eqn:IH} holds for each $n\geq 1$ and $(\phi,\uu),(\phiinv,\uuinv)\in (LX)_n$ with $\|{\ar(\phi)}\|^{(n)}\leq N$. 
Take $(\phi,\uu)\in (LX)_n$ with $\|{\ar(\phi)}\|^{(n)}= N+1$ and its $S$-inverse instruction $(\phiinv,\uuinv)$.

A sketch of the rest of the proof is as follows. We shall show that there exist 
\begin{itemize}
    \item a \emph{sequence} of $(n+1)$-cells in $S$ from $\xi(\phi,\uu)\comp{n-1}{X}\xi(\phiinv,\uuinv)$ to $\id{n}{X}{s^X_{n-1}\xi(\phi,\uu)}$, and 
    \item a \emph{sequence} of $(n+1)$-cells in $S$ from $\xi(\phiinv,\uuinv)\comp{n-1}{X}\xi(\phi,\uu)$ to $\id{n}{X}{t^X_{n-1}\xi(\phi,\uu)}$.
\end{itemize}
Notice that this suffices because $S=\pst(S)$ is closed under compositions.
We only carry out the construction of the former sequence. 
We eventually obtain the sequence
\begin{equation}
\label{eqn:main-proof-structure}
\begin{tikzpicture}[baseline=-\the\dimexpr\fontdimen22\textfont2\relax ]
      \node(a) at (0,3.5) {$\xi(\phi,\uu)\comp{n-1}{X}\xi(\phiinv,\uuinv)$};
      \node(0) at (0,2.5) {$\xi\bigl((\phi,\uu)\comp{n-1}{LX}(\phiinv,\uuinv)\bigr)$};
      \node(1) at (0,1.5) {$\xi(\phi',\uum)$};
      \node(2) at (0,0.5) {$\xi(\phi',\uui)$};
      \node(3) at (0,-0.5) {$\xi\bigl(\delta^{\mainj}(\phi'),\delta^{\mainj}(\uui)\bigr)$};
      \node(b) at (0,-1.5) {$\xi\Bigl(\bigl(\delta^{\mainj}(\phi),\delta^{\mainj}_+(\uu)\bigr)\comp{n-1}{LX}\bigl(\delta^{\maini}(\phiinv),\delta^{\maini}_-(\uuinv)\bigr)\Bigr)$};
      \node(4) at (0,-2.5) {$\xi\bigl(\delta^{\mainj}(\phi),\delta^{\mainj}_+(\uu)\bigr)\comp{n-1}{X}\xi\bigl(\delta^{\maini}(\phiinv),\delta^{\maini}_-(\uuinv)\bigr)$};
      \node(5) at (0,-3.5) {$\id{n}{X}{s^X_{n-1}\xi(\phi,\uu)}$.};
      
      \draw [->] (0,2.2) to node[auto, labelsize] {$w_1$} node[auto,swap,labelsize] {$\sim$} (0,1.8); 
      \draw [->] (0,1.2) to node[auto, labelsize] {$w_2=\xi(\phi'',\uuo)$} (0,0.8); 
      \draw [->] (0,0.2) to node[auto, labelsize] {$w_3$} node[auto,swap,labelsize] {$\sim$}  (0,-0.2); 
      \draw [->] (0,-0.8) to node[auto, labelsize] {$w_4$} node[auto,swap,labelsize] {$\sim$}  (0,-1.2); 
      \draw [->] (0,-2.8) to node[auto, labelsize] {$w_5$} (0,-3.2); 
      \node[rotate=90] at (0.024,3) {$=$};
      \node[rotate=90] at (0.024,-2) {$=$};
\end{tikzpicture}
\end{equation}
Here, $\maini$ and $\mainj$ are suitable integers for which the cells indicated (such as $\delta^{\mainj}_+(\uu)$) are well-defined, and the two equations hold because $\xi$ is a strict $\omega$-functor from $LX$ to $X$.
The $(n+1)$-cells $w_1,w_3$, and $w_4$ are invertible cells induced by coherence (\cref{cor:coherence}), and hence they are in $I\subseteq S$.
The $(n+1)$-cells $w_2$ and $w_5$, on the other hand, are cells in $S=\pst(S)$,
where the latter is produced by the induction hypothesis \cref{eqn:IH} with respect to $N$.

Let 
\[
\ar(\phi)=\kk=
\begin{bmatrix}
    k_0 & & \dots & & k_r\\
    & \underline k_1 & \dots & \underline k_r &
    \end{bmatrix}.
\]
Since $(\phiinv,\uuinv)$ is an $S$-inverse instruction of $(\phi,\uu)$, for each $0\leq i\leq r$ with $k_i=n$, we have $(n+1)$-cells 
\begin{equation}
\label{eqn:omega}
{p}_i\colon u_i\comp{n-1}{X}v_i \to \id{n}{X}{s_{n-1}^X(u_i)}\qquad\text{and}\qquad {q}_i\colon v_i\comp{n-1}{X}u_i \to \id{n}{X}{t_{n-1}^X(u_i)}
\end{equation}
in $S$.
Also notice that $(\phi,\uu)$ and $(\phiinv,\uuinv)$ are $n$-cells in $LX$ composable along the $(n-1)$-dimensional boundary (see \cref{eqn:type-phi-u-phiinv-uinv}), and hence the first equality
\[
\xi(\phi,\uu)\comp{n-1}{X}\xi(\phiinv,\uuinv)=\xi\bigl((\phi,\uu)\comp{n-1}{LX}(\phiinv,\uuinv)\bigr)
\]
in \cref{eqn:main-proof-structure} indeed makes sense (and it holds because $\xi$ is a strict $\omega$-functor).

The $n$-cell $(\phi,\uu)\comp{n-1}{LX}(\phiinv,\uuinv)$ of $LX$ is equal to 
$(\phi\comp{n-1}{L1}\phiinv,\uu\comp{n-1}{TX}\uuinv)$, since both $L!\colon LX\to L1$ and $\ar_X\colon LX\to TX$ are strict $\omega$-functors.
As a pasting diagram, $\uu\comp{n-1}{TX}\uuinv$ is obtained from $\uu$ by replacing, for each $(n-1)$-transversal component $0 \le i \le j \le r$ of $\kk$, the corresponding subsequence
\[
\begin{bmatrix}
u_{i} & & \dots & & u_{j}\\
& \underline u_{i+1} & \dots & \underline u_{j} &
\end{bmatrix}
\]
with
\[
\begin{bmatrix}
u_{i} & & \dots & & u_{j} & & v_{j} & & \dots & & v_{i}\\
& \underline u_{i+1} & \dots & \underline u_{j} & & t_{n-1}^X(u_{j}) & & \underline u_{j} & \dots & \underline u_{i+1} &
\end{bmatrix}.
\]

Now let  $0 \le \maini \le \mainj \le r$ be the leftmost (or first)\footnote{One can choose any $(n-1)$-transversal component here, but this choice simplifies the indices involved.} $(n-1)$-transversal component of $\kk$.
Let $\kkd=\delta^{\mainj}
(\kk\comp{n-1}{T1}\kk)$; that is, $\kkd$ is the pasting scheme obtained from $\kk\comp{n-1}{T1}\kk$ by removing one $n$ in the top row and one $n-1$ in the bottom row from the segment corresponding to the first $(n-1)$-transversal component of $\kk$.
Let $\phi' = \kappa\bigl(\ppair{s_{n-1}^{L1}(\phi),s_{n-1}^{L1}(\phi)},\kkd\bigr)$ and let $\uum$ be the pasting diagram of shape $\kkd$ obtained from 
\[
\uu\comp{n-1}{TX}\uuinv
=
\begin{bmatrix}
u_{0} & & \dots & & u_{\mainj} & & v_{{\mainj}} & & \dots & \\
& \underline u_{1} & \dots & \underline u_{{\mainj}} & & t_{n-1}^X(u_{{\mainj}}) & & \underline u_{{\mainj}} & \dots & 
\end{bmatrix}\footnote{Here we do not write the end of the sequence because it depends on whether $k_r=n$ or not.}
\]
by replacing the segment
\[
\begin{bmatrix}
u_{{\mainj}} & & v_{{\mainj}}\\
& t_{n-1}^X(u_{{\mainj}}) &
\end{bmatrix}
\]
with
\[
\begin{bmatrix}
u_{{\mainj}}\comp{n-1}{X}v_{{\mainj}}\\
\,
\end{bmatrix};
\]
that is, 
\[
\uum=
\begin{bmatrix}
u_{0} & & \dots & & u_{{\mainj}}\comp{n-1}{X}v_{{\mainj}} & & \dots & \\
& \underline u_{1} & \dots & \underline u_{{\mainj}} & & \underline u_{{\mainj}} & \dots & 
\end{bmatrix}.
\]

Let $\uut$ be the pasting diagram of shape $\kkd$ in $LX$ defined as 
\[
\uut=
\begin{bmatrix}
(\widetilde{e}_{k_0},[u_{0}]) & & \dots & & (\widetilde{e}_n\comp{n-1}{L1}\widetilde{e}_n,[u_{{\mainj}}]\comp{n-1}{TX}[v_{{\mainj}}]) & & \dots & \\
& (\widetilde{e}_{\underline k_1},[\underline u_{1}]) & \dots & (\widetilde{e}_{\underline k_{{\mainj}}},[\underline u_{{\mainj}}]) & & (\widetilde{e}_{\underline k_{{\mainj}}},[\underline u_{{\mainj}}]) & \dots &
\end{bmatrix},
\]
where
\[
[u_{\mainj}]\comp{n-1}{TX}[v_{\mainj}]=\begin{bmatrix}
u_{{\mainj}} & & v_{{\mainj}}\\
& t_{n-1}^X(u_{{\mainj}}) &
\end{bmatrix}\in (TX)_n.
\]
Notice that $(\phi',\uut)$ is an $n$-cell of $L^2X$, and that we have 
\begin{align*}
    (L\xi)(\phi',\uut)=\bigl(\phi',(T\xi)(\uut)\bigr)=(\phi',\uum).
\end{align*}
Next we calculate $\mu^L_X(\phi',\uut)$. To this end, we decompose $\uut$ into 
\[
\cchi=
\begin{bmatrix}
\widetilde{e}_{k_0} & & \dots & & \widetilde{e}_n\comp{n-1}{L1}\widetilde{e}_n & & \dots &\\
& \widetilde{e}_{\underline k_1} & \dots & \widetilde{e}_{\underline k_{{\mainj}}} & & \widetilde{e}_{\underline k_{{\mainj}}} & \dots & 
\end{bmatrix}\in(TL1)_n
\]
and 
\[
\uul=
\begin{bmatrix}
[u_{0}] & & \dots & & [u_{{\mainj}}]\comp{n-1}{TX}[v_{{\mainj}}] & & \dots & \\
& [\underline u_{1}] & \dots & [\underline u_{{\mainj}}] & & [\underline u_{{\mainj}}] & \dots & 
\end{bmatrix}\in (T^2X)_n.
\]
Then we have 
\begin{align*}
    \mu^L_X(\phi',\uut)
    &=(\mu^L_1(\phi',\cchi),\mu^T_X(\uul))\\
    &=\left(\mu^L_1(\phi',\cchi),\uu\comp{n-1}{TX}\uuinv\right)\in (LX)_n.
\end{align*}
Now, the $n$-cells $\mu^L_1(\phi',\cchi)$ and $\phi\comp{n-1}{L1}\phiinv$ of $L1$ are parallel. This is because
\begin{align*}
    s^{L1}_{n-1}\bigl(\mu^L_1(\phi',\cchi)\bigr) 
    &= \mu^L_1\bigl(s^{L1}_{n-1}(\phi'),s^{TL1}_{n-1}(\cchi)\bigr)\\ 
    &= \mu^L_1\bigl(s^{L1}_{n-1}(\phi),(T\eta^L_1)\bigl(\ar(s^{L1}_{n-1}\phi)\bigr)\bigr)\\
    &= \mu^L_1\circ(L\eta^L_1)(s^{L1}_{n-1}(\phi))\\
    &= s^{L1}_{n-1}(\phi)\\
    &= s^{L1}_{n-1}\bigl(\phi\comp{n-1}{L1}\phiinv\bigr)
\end{align*}
and similarly 
\begin{align*}
    t^{L1}_{n-1}\bigl(\mu^L_1(\phi',\cchi)\bigr) 
    &= \mu^L_1\bigl(t^{L1}_{n-1}(\phi'),t^{TL1}_{n-1}(\cchi)\bigr)\\ 
    &= \mu^L_1\bigl(s^{L1}_{n-1}(\phi),(T\eta^L_1)\bigl(\ar(s^{L1}_{n-1}\phi)\bigr)\bigr)\\
    &= \mu^L_1\circ(L\eta^L_1)(s^{L1}_{n-1}(\phi))\\
    &= s^{L1}_{n-1}(\phi)\\
    &= t^{L1}_{n-1}(\phiinv)\\
    &= t^{L1}_{n-1}\bigl(\phi\comp{n-1}{L1}\phiinv\bigr).
\end{align*}
Thus by \cref{cor:coherence}, we have an invertible $(n+1)$-cell
\[
w_1\colon\xi\left(\phi\comp{n-1}{L1}\phiinv,\uu\comp{n-1}{TX}\uuinv\right)\to \xi
\left(\mu^L_1(\phi',\cchi),\uu\comp{n-1}{TX}\uuinv\right)
\]
in $X$. Here, the $n$-dimensional domain of $w_1$ is
\begin{align*}
\xi\left(\phi\comp{n-1}{L1}\phiinv,\uu\comp{n-1}{TX}\uuinv\right) &=
    \xi\bigl((\phi,\uu)\comp{n-1}{LX}(\phiinv,\uuinv)\bigr),
\end{align*}
whereas the $n$-dimensional codomain of $w_1$ is
\begin{align*}
    \xi\left(\mu^L_1(\phi',\cchi),\uu\comp{n-1}{TX}\uuinv\right) 
    &=
    \xi\circ\mu^{L}_X(\phi',\uut)\\
    &=\xi\circ(L\xi)(\phi',\uut)\\
    &=\xi(\phi',\uum),
\end{align*}
as indicated in \cref{eqn:main-proof-structure}.

To define $w_2$, let $\kkdd$ be the pasting scheme obtained from $\kkd$ by replacing the middle $n$ in the first $(n-1)$-transversal component by $n+1$, and let $\phi'' = \kappa(\ppair{\phi',\phi'},\kkdd)$.
Combining this $\phi''$ with the pasting diagram 
\[
\uuo=
\begin{bmatrix}
u_{0} & & \dots & & {p}_{\mainj} & & \dots & \\
& \underline u_{1} & \dots & \underline u_{{\mainj}} & & \underline u_{{\mainj}} & \dots & 
\end{bmatrix},
\]
which is obtained from $\uum$ by replacing $u_{{\mainj}}\comp{n-1}{X}v_{{\mainj}}$ with the $(n+1)$-cell
\[
{p}_{\mainj}\colon u_{{\mainj}}\comp{n-1}{X}v_{{\mainj}} \to \id{n}{X}{s_{n-1}^X(u_{{\mainj}})}
\]
in $S$ as in \cref{eqn:omega}, we obtain
\[
w_2=\xi(\phi'',\uuo) \colon \xi(\phi',\uum) \to \xi(\phi',\uui)
\]
in $\pst(S)$, where 
\[
\uui=
\begin{bmatrix}
u_{0} & & \dots & & \id{n}{X}{s_{n-1}^X(u_{{\mainj}})} & & \dots & \\
& \underline u_{1} & \dots & \underline u_{{\mainj}} & & \underline u_{{\mainj}} & \dots & 
\end{bmatrix}.
\]

By \cref{lem:unit-law}, we obtain an invertible $(n+1)$-cell
\[
w_3 \colon \xi(\phi',\uui) \to \xi\bigl(\delta^{\mainj}(\phi'),\delta^{\mainj}(\uui)\bigr)
\]
in $X$.

Next, to define $w_4$, we show that the $n$-cells $\delta^{\mainj}(\phi')$ and $\delta^{\mainj}(\phi)\comp{n-1}{L1}\delta^{\maini}(\phiinv)$ of $L1$ are parallel. Indeed, we have
\begin{align*}
    s^{L1}_{n-1}\bigl(\delta^{\mainj}(\phi')\bigr) 
    &=s^{L1}_{n-1}(\phi')\\
    &=s^{L1}_{n-1}(\phi)\\
    &=s^{L1}_{n-1}\bigl(\delta^{\mainj}(\phi)\bigr)\\
    &=s^{L1}_{n-1}\bigl(\delta^{\mainj}(\phi)\comp{n-1}{L1}\delta^{\maini}(\phiinv)\bigr)
\end{align*}
and 
\begin{align*}
    t^{L1}_{n-1}\bigl(\delta^{\mainj}(\phi')\bigr) 
    &=t^{L1}_{n-1}(\phi')\\
    &=s^{L1}_{n-1}(\phi)\\
    &=t^{L1}_{n-1}(\phiinv)\\
    &=t^{L1}_{n-1}\bigl(\delta^{\maini}(\phiinv)\bigr)\\
    &=t^{L1}_{n-1}\bigl(\delta^{\mainj}(\phi)\comp{n-1}{L1}\delta^{\maini}(\phiinv)\bigr).
\end{align*}
Moreover, we have 
\[
\delta^{\mainj}(\uui)= \delta^{\mainj}_+(\uu)\comp{n-1}{TX}\delta^{\maini}_-(\uuinv).
\]
Therefore by \cref{cor:coherence}, we obtain an invertible $(n+1)$-cell
\[
w_4\colon \xi\bigl(\delta^{\mainj}(\phi'),\delta^{\mainj}(\uui)\bigr)\to \xi\bigl(\delta^{\mainj}(\phi)\comp{n-1}{L1}\delta^{\maini}(\phiinv),\ \delta^{\mainj}_+(\uu)\comp{n-1}{TX}\delta^{\maini}_-(\uuinv)\bigr)
\]
in $X$. The $n$-dimensional codomain of $w_4$ is 
\begin{align*}
    \xi\bigl(\delta^{\mainj}(\phi)\comp{n-1}{L1}\delta^{\maini}(\phiinv),\ \delta^{\mainj}_+(\uu)\comp{n-1}{TX}\delta^{\maini}_-(\uuinv)\bigr)
    &=\xi\Bigl(\bigl(\delta^{\mainj}(\phi),\delta^{\mainj}_+(\uu)\bigr)\comp{n-1}{LX}\bigl(\delta^{\maini}(\phiinv),\delta^{\maini}_-(\uuinv)\bigr)\Bigr),
\end{align*}
as indicated in \cref{eqn:main-proof-structure}.

The second equation in \cref{eqn:main-proof-structure} holds since $\xi$ is a strict $\omega$-functor.

Since $\|{\ar(\delta^{\mainj}(\phi))}\|^{(n)}=N$ and $\bigl(\delta^{\maini}(\phiinv),\delta^{\maini}_-(\uuinv)\bigr)$ is an $S$-inverse instruction of $\bigl(\delta^{\mainj}(\phi),\delta^{\mainj}_+(\uu)\bigr)$, we obtain an $(n+1)$-cell
\[
w_5\colon \xi\bigl(\delta^{\mainj}(\phi),\delta^{\mainj}_+(\uu)\bigr)\comp{n-1}{X}\xi\bigl(\delta^{\maini}(\phiinv),\delta^{\maini}_-(\uuinv)\bigr)\to 
\id{n}{X}{s^X_{n-1}\xi(\delta^{\mainj}(\phi),\delta^{\mainj}_+(\uu))}
\]
in $S$, by the induction hypothesis \cref{eqn:IH}.
Because we have 
\[s^{L1}_{n-1}(\delta^{\mainj}(\phi))=s^{L1}_{n-1}(\phi)\quad\text{and}\quad s^{TX}_{n-1}(\delta^{\mainj}_+(\uu))=s^{TX}_{n-1}(\uu),\] 
the $n$-dimensional codomain of $w_5$ is $\id{n}{X}{s^X_{n-1}\xi(\phi,\uu)}$, as indicated in \cref{eqn:main-proof-structure}.
\end{proof}

\begin{corollary}
\label{cor:eq-rel}
    Let $X$ be a weak $\omega$-category. Then $\sim$ is an equivalence relation on the set of cells of $X$. 
\end{corollary}

\begin{corollary}
\label{cor:unique-inv}
    Let $X$ be a weak $\omega$-category, $n\geq 1$, $u\colon x\to y$ be an invertible $n$-cell in $X$ and $v,v'\colon y\to x$ be inverses of $u$. Then we have $v\sim v'$.
\end{corollary}
\begin{proof}
    By \cref{lem:unit-law}, we have $v\sim v\comp{n-1}{}\id{n}{}{x}$. Thus we have 
    \begin{align*}
        v &\sim v\comp{n-1}{}\id{n}{}{x}\\
        &\sim v\comp{n-1}{}(u\comp{n-1}{}v')\\
        &\sim (v\comp{n-1}{}u)\comp{n-1}{}v'\\
        &\sim \id{n}{}{y}\comp{n-1}{}v'\\
        &\sim v'.\qedhere
    \end{align*}
\end{proof}

\begin{corollary}
\label{cor:invariance}
    Let $X$ be a weak $\omega$-category, $n\geq 1$, and $u,v\colon x\to y$ be a parallel pair of $n$-cells in $X$ such that $u \sim v$.
    Suppose that $u$ is invertible.
    Then $v$ is invertible too.
\end{corollary}
\begin{proof}
    Since $u$ is invertible, it has an inverse $w$ with $u\comp{n-1}{}w \sim \id{n}{}{x}$ and $w\comp{n-1}{}u \sim \id{n}{}{y}$.
    Thus we have
    \[
    v\comp{n-1}{}w \sim u\comp{n-1}{}w \sim \id{n}{}{x}
    \]
    and similarly $w\comp{n-1}{}v \sim \id{n}{}{y}$.
\end{proof}

\subsection{The core weak \texorpdfstring{$\omega$}{omega}-groupoid of a weak \texorpdfstring{$\omega$}{omega}-category}
We conclude this paper with the construction of the \emph{core weak $\omega$-groupoid} of a weak $\omega$-category.
Let $X$ be a weak $\omega$-category and $n\geq 0$. An $n$-cell $x\in X_n$ is \emph{hereditarily invertible} if either
\begin{itemize}
    \item $n=0$, or 
    \item $n\geq 1$, $x$ is invertible, and $s_{n-1}^X(x)$ and $t_{n-1}^X(x)$ are hereditarily invertible.
\end{itemize}
By definition, the set of all hereditarily invertible cells of $X$ is closed under $s^X_{n}$ and $t^X_{n}$, and hence forms a globular subset $k(X)$ of $X$.
Moreover, $k(X)$ is closed under pasting:

\begin{proposition}
\label{prop:hereditarily-invertible}
    Let $(X,\xi)$ be a weak $\omega$-category, $n\in\mathbb{N}$ and $(\phi,\uu)\in 
    \bigl(L\bigl(k(X)\bigr)\bigr)_n\subseteq (LX)_n$. Then $\xi(\phi,\uu)\in k(X)_n$. 
    Therefore $k(X)$ is also a weak $\omega$-category, which is a weak $\omega$-subcategory (i.e., subobject in $\WkCats{\omega}$) of $X$.
\end{proposition}
Note that, since $L$ preserves pullbacks, it also preserves monomorphisms.
Therefore it makes sense to regard $L\bigl(k(X)\bigr)$ as a globular subset of $L(X)$.
\begin{proof}[Proof of \cref{prop:hereditarily-invertible}]
    We prove this by induction on $n$.
    The base case $n=0$ is clear, so let $m\geq 1$ and suppose that the claim holds when $n=m-1$. Take any $(\phi,\uu)\in \bigl(L\bigl(k(X)\bigr)\bigr)_m$ and let $x=\xi(\phi,\uu)$. Then $x$ is invertible by \cref{thm:main}. Moreover,
    $s^X_{m-1}(x)=\xi\bigl(s^{L1}_{m-1}(\phi),s^{TX}_{m-1}(\uu)\bigr)$
    is hereditarily invertible by the induction hypothesis since $\bigl(s^{L1}_{m-1}(\phi),s^{TX}_{m-1}(\uu)\bigr)\in \bigl(L\bigl(k(X)\bigr)\bigr)_{m-1}$. Similarly, $t^X_{m-1}(x)$ is hereditarily invertible. 
\end{proof}

We define a \emph{weak $\omega$-groupoid} to be a weak $\omega$-category in which every cell of dimension $\geq 1$ is invertible, or equivalently every cell is hereditarily invertible. 
Let $\WkGpds{\omega}$ be the full subcategory of $\WkCats{\omega}$ consisting of all weak $\omega$-groupoids. Since every strict $\omega$-functor preserves (hereditarily) invertible cells, we see that any strict $\omega$-functor $X\to Y$ from a weak $\omega$-groupoid $X$ to a weak $\omega$-category $Y$ factors through the inclusion $k(Y)\to Y$. Since $k(Y)$ is a weak $\omega$-groupoid, $k\colon \WkCats{\omega}\to\WkGpds{\omega}$ is the right adjoint of the inclusion $\WkGpds{\omega}\to \WkCats{\omega}$.

\begin{remark}
\label{rmk:weak-omega-functors}
    We remark that $k$ also gives rise to the right adjoint of the inclusion functor $\WkGpd{\omega}\to\WkCat{\omega}$, where $\WkCat{\omega}$ is the category (defined in \cite{Garner_homomorphisms}) of weak $\omega$-categories and \emph{weak $\omega$-functors}, and $\WkGpd{\omega}$ is the full subcategory of $\WkCat{\omega}$ consisting of all weak $\omega$-groupoids. To show this, it is essentially enough to observe that weak $\omega$-functors ``preserve'' invertible cells (and hence also hereditarily invertible ones). However, the latter statement perhaps needs some clarification, since a weak $\omega$-functor $X\to Y$ does not induce a globular map between the underlying globular sets of $X$ and $Y$ in general.

    We recall that if $X$ and $Y$ are weak $\omega$-categories, then a weak $\omega$-functor $X\to Y$ defined in \cite{Garner_homomorphisms} comes equipped with a span 
    \begin{equation}
    \label{eqn:span}
\begin{tikzpicture}[baseline=-\the\dimexpr\fontdimen22\textfont2\relax ]
      \node(00) at (0,-0.75) {$X$};
      \node(01) at (3,-0.75) {$Y$};
      \node(10) at (1.5,0.75) {$QX$};
      
      \draw [->] (10) to node[auto, swap,labelsize] {$\varepsilon_X$} (00); 
      \draw [->] (10) to node[auto, labelsize] {$f$} (01); 
\end{tikzpicture}
\end{equation}
    of strict $\omega$-functors, where $\varepsilon_X$ is moreover a \emph{trivial fibration} in the sense that its underlying map between globular sets has the right lifting property with respect to $\iota_n\colon\partial G^n\to G^n$ for all $n\geq 0$. ($Q$ is in fact a comonad on $\WkCats{\omega}$, and $\WkCat{\omega}$ is defined as the Kleisli category of $Q$.) Since every trivial fibration is contractible, we see by \cref{prop:contractible-reflect-invertible,lem:morphism-preserves-invertible-cells} that for any invertible $n$-cell $u$ of $X$, every $n$-cell $\overline{u}$ of $QX$ such that $\varepsilon_X(\overline{u})=u$ is invertible, and so is the $n$-cell $f(\overline{u})$ of $Y$. (Note also that for any cell $u$ of $X$ there exists some cell $\overline{u}$ of $QX$ with $\varepsilon_X(\overline{u})=u$, since a trivial fibration is surjective.) This is what we mean by ``weak $\omega$-functors preserve invertible cells.''

    In order to show that $k$ is the right adjoint of the inclusion $\WkGpd{\omega}\to\WkCat{\omega}$, observe that if $X$ is a weak $\omega$-groupoid, then so is $QX$ (by the presence of a trivial fibration $\varepsilon_X$), and hence the right leg $f$ of the span \cref{eqn:span} factors through $k(Y)$. Therefore the weak $\omega$-functors $X\to Y$ correspond to the weak $\omega$-functors $X\to k(Y)$.
\end{remark}

A completely parallel argument works more generally for $(\infty,n)$-categories with $n\in\mathbb{N}$ in place of weak $\omega$-groupoids, where we say that a weak $\omega$-category is an \emph{$(\infty,n)$-category} if every cell of dimension $>n$ is invertible. 
Thus we can also obtain the \emph{core $(\infty,n)$-category} of any weak $\omega$-category.

\subsection*{Acknowledgements}
The first author acknowledges the support of JSPS Overseas Research Fellowships and an Australian Research Council Discovery Project DP190102432.
The second author is supported by JSPS Research Fellowship for Young Scientists and JSPS KAKENHI Grant Number JP23KJ1365.
The third author was supported by JSPS KAKENHI Grant Number JP21K20329 and JP23K12960.
\bibliographystyle{plain}
\bibliography{mybib}
\end{document}